\documentclass[12pt]{amsart}

\usepackage[a4paper,margin=1in]{geometry}
\usepackage{graphicx}
\usepackage{amsmath,amsthm,amssymb,amsfonts,mathrsfs,mathtools,dsfont}
\usepackage{xcolor}
\usepackage{enumitem}
\usepackage{microtype}
\usepackage[colorlinks=true,linkcolor=blue!60!black,citecolor=blue!60!black,urlcolor=blue!60!black]{hyperref}
\usepackage[nameinlink,capitalize,noabbrev]{cleveref}

\numberwithin{equation}{section}

\theoremstyle{plain}
\newtheorem{theorem}{Theorem}[section]
\newtheorem{lemma}[theorem]{Lemma}
\newtheorem{cor}[theorem]{Corollary}
\newtheorem{prop}[theorem]{Proposition}

\newtheorem{thmA}{Theorem}

\newtheorem{thmB}{Theorem}

\newtheorem{assumpA}{Assumption}

\newtheorem{corC}{Corollary}

\theoremstyle{definition}
\newtheorem{definition}[theorem]{Definition}
\newtheorem{example}[theorem]{Example}
\newtheorem{setting}[theorem]{Setting}

\theoremstyle{remark}

\newtheorem*{introremark}{Remark}

\newcommand{\psh}{\mathrm{PSH}}
\newcommand{\e}{\mathrm{e}}
\newcommand{\ddc}{dd^c}

\title{Regularity of Solutions to Monge--Amp\`ere Equations on Stein Spaces}
\author{Hongyu Chen}
\thanks{School of Mathematics, Sichuan University, Chengdu, China. Email: \texttt{hongyu.chern@gmail.com}.}

\begin{document}

\begin{abstract}
In this note, we study degenerate complex Monge--Amp\`ere equations on singular Stein spaces with right-hand sides depending on the unknown. First, we prove continuity up to the boundary for solutions. Next, under a holomorphic peak-set condition, we prove local H\"older estimates on the regular locus, allowing the boundary datum to fail to be H\"older continuous at the singular boundary. This includes the case of finite singular boundary sets.
\end{abstract}

\maketitle

\section*{Introduction}
The Dirichlet problem for complex Monge--Amp\`ere equations is classical in
the smooth setting. On smoothly bounded strongly pseudoconvex domains in
$\mathbb C^n$, Caffarelli--Kohn--Nirenberg--Spruck
\cite{CaffarelliKohnNirenbergSpruck1985} proved the existence of a unique
smooth strictly plurisubharmonic solution for smooth positive right-hand
sides, while Ko{\l}odziej~\cite{Kolodziej1998} obtained continuous solutions
for $L^p$ densities, for some $p>1$.

On singular Stein spaces, the corresponding theory is much less developed.
Guedj--Guenancia--Zeriahi~\cite{GuedjGuenanciaZeriahi2023} proved a general
continuity result for local Dirichlet problems. Motivated by their work, we
consider a local Dirichlet problem whose right-hand side depends on the unknown
function.

Let $X$ be a reduced, locally irreducible Stein space of dimension $n\ge 1$.
We denote by $X_{\rm reg}$ the regular locus of $X$ and by
$X_{\rm sing}:=X\setminus X_{\rm reg}$ its singular locus. Let
$\Omega\subset X$ be a relatively compact strongly pseudoconvex domain, and
fix a smooth positive volume form $dV_X$ on $X_{\rm reg}$, extended as a
measure by zero across $X_{\rm sing}$. Given
$\phi\in C^0(\partial\Omega)$, we consider the Dirichlet problem
\begin{equation}\label{main-eq-general}
\left\{
\begin{aligned}
(\ddc u)^n &= F(u,x)dV_X \qquad &&\text{in } \Omega,\\
u|_{\partial\Omega} &= \phi \qquad &&\text{on } \partial\Omega.
\end{aligned}
\right.
\end{equation}
The equation is understood in the sense of Bedford--Taylor
\cite{BedfordTaylor1976,BedfordTaylor1982}. We impose the following assumption
on $F$.

\begin{assumpA}\label{ass:F}
Let $M:=\max_{\partial\Omega}\phi.$
Assume that the function $F: ]-\infty,M]\times \Omega \to \mathbb{R}_{\ge 0}$
satisfies the following properties:
\begin{itemize}
    \item[(F1)] the map $(t,x)\mapsto F(t,x)$ is measurable;
    \item[(F2)] for each fixed $x\in\Omega$, the map
    $
    t\mapsto F(t,x)
    $
    is continuous and nondecreasing on $]-\infty,M]$;
    \item[(F3)] there exists a function $g\in L^p(\Omega, dV_X)$ for some $p>1$ such that
    $$
    0\le F(t,x)\le g(x)
    \qquad \text{for all } (t,x)\in ]-\infty,M]\times \Omega.
    $$
\end{itemize}
\end{assumpA}

Assuming that $F$ satisfies Assumption~\ref{ass:F}, our first result is the following.
\begin{thmA}\label{main thm}
    The Dirichlet problem \eqref{main-eq-general} admits a unique solution $u\in \psh(\Omega) \cap C^0(\overline{\Omega})$.
\end{thmA}

\begin{introremark}
In view of {\rm(F2)}, condition {\rm(F3)} is equivalent to requiring that $
F(M,\cdot)\in L^p(\Omega, dV_X)$
for some $p>1$. 
\end{introremark}

This framework includes the classical prescribed density equation $F(u,x)=f(x)$ and the Aubin--Yau type equation $F(u,x)=\e^{\lambda u} f(x)$, where $f\in L^p(\Omega,dV_X)$, $p>1$, and $\lambda>0$,  as studied for instance in \cite{BoucksomEyssidieuxGuedjZeriahi2010,Berman2019}.

We next turn to local H\"older regularity. On bounded domains in $\mathbb C^n$, H\"older regularity for solutions with $L^p$ densities has been studied in
\cite{GuedjKolodziejZeriahi2008,Charabati2015,Charabati2017,
BaraccoKhanhPintonZampieri2016}. Analogous results on compact complex manifolds were obtained in
\cite{DDGPKZ14,KolodziejNguyen2018,LuPhungTo2021,DinhKolodziejNguyen2022}.
More recently, Cerqueira \cite{CerqueiraGoncalves2026} studied the modulus of continuity for complex Monge--Amp\`ere equations on Stein spaces with isolated interior singularities.

The purpose of the present paper is to study a boundary version of this question on singular Stein spaces. We allow
$S:=\partial\Omega\cap X_{\rm sing}$
to be nonempty and assume that the boundary datum is only locally H\"older continuous along $\partial\Omega\cap X_{\rm reg}$.

Here and throughout, $q$ denotes the conjugate exponent of $p$, that is,
$\frac{1}{p}+\frac{1}{q}=1$. 
Moreover, $\alpha$ denotes a H\"older exponent with $0<\alpha\le 1$. Fix a smooth positive
$(1,1)$-form $\beta$ on $X$, not necessarily closed. We denote by
$d_\beta(\cdot,\cdot)$ the distance induced by $\beta$. 

Recall that a
compact set $E\subset\partial\Omega$ is called a holomorphic peak set for
$\Omega$ if there exists a holomorphic function $P$ in a neighborhood of
$\overline\Omega$ such that
$P=1$ on $E$,
$|P|<1 $ on  $\overline\Omega\setminus E$.

\begin{thmB}
\label{0_thm:singular-boundary-holder}
Let $u$ be the solution given by Theorem~\ref{main thm}. Assume that
$\phi$ is locally $\alpha$-H\"older continuous on
$\partial\Omega\cap X_{\rm reg}$ and
$S$
is a holomorphic peak set for $\Omega$.  Assume moreover that the local $\alpha$-H\"older constants of
$\phi$ blow up at most polynomially near $S$: there exist
constants $C>0$, $\sigma\ge0$, and $\varepsilon_0>0$ such that
$$
|\phi(\xi)-\phi(\eta)|
\le
C\,\max\{d_\beta(\xi,S),d_\beta(\eta,S)\}^{-\sigma}
d_\beta(\xi,\eta)^\alpha
$$
whenever $\xi,\eta\in\partial\Omega\cap X_{\rm reg}$ and $0<\max\{d_\beta(\xi,S),d_\beta(\eta,S)\}<\varepsilon_0$.
Then $u$ is locally $\alpha_*$-H\"older continuous on
$\Omega\cap X_{\rm reg}$ for every
$0<\alpha_*<
\min\left\{
\frac{1}{nq+1},
\frac{\alpha}{2(\sigma+1)}
\right\}
$.
\end{thmB}

 In Section~\ref{peak function}, we prove that every finite singular
boundary set is a holomorphic peak set. We also give examples showing that  singular boundary sets may have positive dimension and still be holomorphic peak sets. 

\begin{corC}\label{finite-cor}
The conclusion of Theorem~\ref{0_thm:singular-boundary-holder} remains valid if the holomorphic peak-set assumption on $S$ is replaced by the assumption that $S$ is finite.
\end{corC}

\begin{introremark}
When $\sigma=0$, the assumption is essentially a uniform
$\alpha$-H\"older condition on $\partial\Omega$. When
$0<\sigma<\alpha$, it allows the boundary datum to have only
$(\alpha-\sigma)$-H\"older regularity at the singular boundary points.
In this case, the boundary contribution to our exponent,
$\min\left\{
\frac{1}{nq+1},\,
\frac{\alpha}{2(\sigma+1)}
\right\}$,
is larger than the exponent $\min\left\{
\frac{1}{nq+1},\,
\frac{\alpha-\sigma}{2}
\right\}$ obtained by treating $\phi$ merely as a
globally $(\alpha-\sigma)$-H\"older function, as in \cite{CerqueiraGoncalves2026}.
When $\sigma\ge\alpha$, no positive H\"older regularity at $S$ is imposed.
In particular, logarithmic moduli of continuity at $S$ are allowed.
\end{introremark}

Theorem~\ref{0_thm:singular-boundary-holder} applies to boundary data which are not H\"older continuous at the
singular boundary points. For instance, when $S=\{a\}$, let $D:=\max_{x,y\in\overline\Omega} d_\beta(x,y)$
and define
$$
\phi(x):=\frac{1}{\log\bigl(2D/d_\beta(x,a)\bigr)} \quad \bigl(x\in\partial\Omega\setminus\{a\}\bigr), \qquad \phi(a):=0
$$
As verified in Example~\ref{log_example},  this function satisfies the assumptions of
Theorem~\ref{0_thm:singular-boundary-holder}  in the case $\sigma\ge \alpha$. However, it is not H\"older continuous at $a$. 
Theorem~\ref{0_thm:singular-boundary-holder}  still yields a local
 H\"older continuity for the corresponding solution on
$\Omega\cap X_{\rm reg}$.

The paper is organized as follows. Section~1 recalls the necessary preliminaries. Section~2 proves Theorem~\ref{main thm}. Section~3 establishes H\"older regularity near the boundary. Section~4 proves Theorem~\ref{0_thm:singular-boundary-holder} and a complementary logarithmic modulus of continuity result.

\medskip \noindent\textit{Acknowledgements.}---
This work was carried out during the author's visit to the Institut de
Math\'ematiques de Toulouse. The author is deeply grateful to Vincent Guedj
for his guidance and valuable comments, to Ahmed Zeriahi and Guilherme
Cerqueira-Gon\c{c}alves for helpful discussions, and to his advisors, An-Min Li
and Li Sheng, for their encouragement and financial support.

\section{Preliminaries}

\subsection{Plurisubharmonic functions on complex spaces.}--Throughout this paper, we let $X$ be a reduced complex analytic space of pure dimension $n\ge 1$, and denote by $X_{\mathrm{reg}}$ the complex manifold of regular points of $X$, and by
$X_{\mathrm{sing}}:=X\setminus X_{\mathrm{reg}}$ its singular subset of complex codimension $\ge 1$.

We briefly recall some standard notions on complex spaces. By definition, for every point $x_{0}\in X$, there exists an open neighborhood $U$ of $x_{0}$, an integer $N\ge 1$, and a holomorphic embedding
$j:U\hookrightarrow \mathbb{C}^{N}$ such that $j(U)$ is an analytic subset of $\mathbb{C}^{N}$. Using these local embeddings, one defines smooth differential forms on $X$ as smooth forms on $X_{\mathrm{reg}}$ which, locally on $U$, are induced by the restriction of smooth ambient forms on $\mathbb{C}^{N}$. Currents on $X$ are then defined by duality through their action on compactly supported smooth forms. In particular, the operators $\partial$, $\bar\partial$, $d$, $d^{c}$, and $dd^{c}$ are well defined in the sense of currents (see \cite{Demailly1985} for more details). Similarly, one can introduce holomorphic and plurisubharmonic functions on complex spaces by means of these local embeddings.

\begin{definition}
Let $u:X\to \mathbb{R}\cup\{-\infty\}$ be a given function.

\begin{enumerate}
\item We say that $u$ is plurisubharmonic on $X$ if it is locally the restriction of a plurisubharmonic function on a local embedding of $X$ onto an analytic subset of $\mathbb{C}^N$.

\item We say that $u$ is weakly plurisubharmonic on $X$ if $u$ is locally bounded from above on $X$ and its restriction to the complex manifold $X_{\mathrm{reg}}$ is plurisubharmonic.
\end{enumerate}
\end{definition}

We also recall a useful characterization due to Forn\ae ss and Narasimhan \cite[Theorem 5.3.1]{FornaessNarasimhan1980}: a function $u$ is plurisubharmonic on $X$ if and only if for every analytic disc $h:\mathbb{D}\to X$, the composition $u\circ h$ is subharmonic or identically equal to $-\infty$.

If $u$ is weakly plurisubharmonic on $X$, then its restriction to the complex manifold $X_{\mathrm{reg}}$ is plurisubharmonic, hence upper semicontinuous on $X_{\mathrm{reg}}$. It is therefore natural to extend $u$ to all of $X$ by setting
$$
u^*(x):=\limsup_{X_{\mathrm{reg}}\ni y\to x}u(y),\qquad x\in X.
$$
The function $u^{*}$ is upper semicontinuous, locally integrable on $X$, and satisfies
$
dd^{c}u^{*}\ge 0
$
in the sense of currents on $X$; see \cite[Th\'eor\`eme 1.7]{Demailly1985}. Moreover, when $X$ is locally irreducible, the two notions introduced above coincide, by \cite[Th\'eor\`eme 1.10]{Demailly1985}.    

\begin{theorem}\label{weak_psh_psh}
Let $X$ be a locally irreducible analytic space, and let $u:X\to \mathbb{R}\cup\{-\infty\}$ be a weakly plurisubharmonic function on $X$. Then the function $u^{*}$  is plurisubharmonic on $X$.
\end{theorem}

Since $u$ is plurisubharmonic on $X_{\mathrm{reg}}$, we have
$u^{*}=u$ on $X_{\mathrm{reg}}.$
Hence $u^{*}$ is the upper semicontinuous extension of $u|_{X_{\mathrm{reg}}}$ to $X$. We denote by $\psh(X)$ the set of plurisubharmonic functions on $X$. 
We record the following useful consequence.

\begin{cor}\label{cor of family}
Let $\mathcal{U}\subset \psh(X)$ be a nonempty family of plurisubharmonic functions which is locally uniformly bounded from above on $X$. Then its upper envelope
$$
U:=\sup\{u \,;\, u\in \mathcal{U}\}
$$
is a well-defined Borel function, and its upper semicontinuous regularization $U^{*}$ is plurisubharmonic on $X$.
\end{cor}

 A subset $\Omega\subset X$ is said to be relatively compact in $X$ if its closure
$\overline{\Omega}$ is compact in $X$. In this case, we write $$
\Omega\Subset X.
$$
Following \cite[Theorem 6.1]{FornaessNarasimhan1980}, we say that $X$ is Stein if it admits a $C^{2}$-smooth strongly plurisubharmonic exhaustion. We will use the following definition:
\begin{definition}
A  domain $\Omega\Subset X$ is strongly pseudoconvex if it admits a negative $C^{2}$-smooth strongly plurisubharmonic exhaustion, i.e. a function $\rho$ in a neighborhood $\Omega'$ of $\overline{\Omega}$ such that
$
\Omega=\{x\in \Omega' \,;\, \rho(x)<0\},
$
and, for every $c<0$, the sublevel set
$\{x\in \Omega' \,;\, \rho(x)<c\}$
is relatively compact in $\Omega$.
\end{definition}

\subsection{Comparison principle}--- 
On complex spaces, the complex Monge--Amp\`ere operator was introduced and studied in \cite{Bedford1982,Demailly1985}. In this setting, if
$u\in \psh(X)\cap L^\infty_{\mathrm{loc}}(X),$
then the Monge--Amp\`ere measure
$(\ddc u)^n$
is well defined on the regular part $X_{\mathrm{reg}}$ and extends to a Borel measure on $X$ by assigning zero mass to $X_{\mathrm{sing}}$. This construction extends the foundational Bedford--Taylor theory \cite{BedfordTaylor1976,BedfordTaylor1982} to singular complex spaces.

As a consequence, several standard properties of the complex Monge--Amp\`ere operator on $\psh(X)\cap L^\infty_{\mathrm{loc}}(X)$ remain valid in singular complex spaces. In particular, we recall the following comparison principle in  \cite[Theorem 4.1]{Bedford1982} (see also \cite[Theorem 3.29, Corollary 3.30]{GuedjZeriahi2017}).  For convenience, we set
$
\psh^\infty(X):=\psh(X)\cap L^\infty(X).
$

\begin{prop}\label{comparison}
Let $\Omega\Subset X$ be an open set and $u,v\in \psh^{\infty}(\Omega)$. Assume that $\lim \inf_{x\to \zeta} \bigl( u(x) -v(x)\bigr)\ge 0$ for any $\zeta \in \partial\Omega$. Then 
$$
\int_{\{u<v\}}(\ddc v)^n \le \int_{\{u<v\}}(\ddc u)^n. 
$$
In particular, if $(\ddc u)^n \le (\ddc v)^n $ weakly on $\Omega$, then $v\le u$ on $\Omega$.
    
\end{prop}

We next prove a comparison lemma tailored to our Dirichlet problem \eqref{main-eq-general}.
\begin{lemma}\label{0-1}
    Let $F_1, F_2$ satisfy Assumption \ref{ass:F} and  $0\le F_1(t,x) \le F_2(t,x)$. Let $w_1,w_2\in \psh^\infty(\Omega)$ such that
$(\ddc w_1)^n\le F_1(w_1,x)dV_X$ and
$(\ddc w_2)^n\ge F_2(w_2,x)dV_X$ in $\Omega$.
    If $w_2 \le w_1$ on $\partial \Omega$, then $w_2 \le w_1$ in $\Omega$.
\end{lemma}

\begin{proof}

By \cite[Theorem A]{GuedjGuenanciaZeriahi2023}, there exists a unique solution $\psi\in \psh(\Omega)\cap C^0(\overline{\Omega})$ to the Dirichlet problem $(\ddc \psi)^n = dV_X$ in $\Omega$ and $\psi|_{\partial\Omega} = 0$ on $\partial \Omega$.
By the maximum principle, we have $\psi\le 0$ on $\Omega$. For any $\varepsilon>0$, set
\begin{equation}\label{def-w2e}
w_{2,\varepsilon}:=w_2+\varepsilon\psi.
\end{equation}
Since $\psi\le 0$ on $\Omega$, it follows that $w_{2,\varepsilon}\le w_2$, and 
therefore
$
\{w_1<w_{2,\varepsilon}\}\subset \{w_1<w_2\}.
$
The comparison principle Proposition \ref{comparison} yields
\begin{equation}
    \int_{\{w_1<w_{2,\varepsilon}\}} (\ddc w_{2,\varepsilon})^n \le   \int_{\{w_1<w_{2,\varepsilon}\}} (\ddc w_1)^n \le \int_{\{w_1<w_{2,\varepsilon}\}} F_1(w_1,x)dV_X.
\end{equation}
Since $(\ddc w_{2,\varepsilon})^n \ge  (\ddc w_{2})^n + \varepsilon^n (\ddc\psi)^n $, we obtain
\begin{align*}
      \int_{\{w_1<w_{2,\varepsilon}\}}F_2(w_2,x)  dV_X +   \varepsilon^n  \int_{\{w_1<w_{2,\varepsilon}\}}dV_X &\le \int_{\{w_1<w_{2,\varepsilon}\}} (\ddc w_{2})^n  +   \varepsilon^n  \int_{\{w_1<w_{2,\varepsilon}\}} (\ddc\psi)^n \\
      &\le \int_{\{w_1<w_{2,\varepsilon}\}} F_1(w_1,x)dV_X.  
\end{align*}
Since  $t\mapsto F_1(t,x)$ is nondecreasing, $0\le F_1\le F_2$ and $w_1<w_{2,\varepsilon} < w_2$ on $\{w_1<w_{2,\varepsilon}\}$, we have
\begin{equation}
     \int_{\{w_1<w_{2,\varepsilon}\}}F_1(w_1,x)  dV_X+   \varepsilon^n  \int_{\{w_1<w_{2,\varepsilon}\}} dV_X  \le \int_{\{w_1<w_{2,\varepsilon}\}} F_1(w_1,x)dV_X. 
\end{equation}
It follows that the set $\{w_1<w_{2,\varepsilon}\}$ has Lebesgue measure zero. Since $\{w_1<w_{2}\}= \bigcup_{j\ge1} \{w_1<w_{2,\frac{1}{j}}\}$, it follows that the set $\{w_1<w_{2}\}$ also has Lebesgue measure zero. Hence, $w_2\le w_1$ on $\Omega$ by the sub-mean value inequalities.
\end{proof}

\section{Dirichlet problem on singular complex spaces}Throughout the remainder of this paper, we fix a smooth positive $(1,1)$-form
$\beta$ and use $\beta^n$ as a reference volume form. Write
$dV_X=h\,\beta^n$ on $\Omega\cap X_{\mathrm{reg}}$, where $h$ is smooth and
positive. We extend both measures by zero across $X_{\mathrm{sing}}$. In local
embeddings, smooth positive volume forms are mutually comparable. Hence, since
$\Omega\Subset X$, the function $h$ is bounded on $\Omega\cap X_{\mathrm{reg}}$.
Thus $(dd^c u)^n=F(u,x)\,dV_X$
is equivalent to
$$
(dd^c u)^n=\widetilde F(u,x)\,\beta^n,\qquad
\widetilde F(t,x):=h(x)F(t,x).
$$
If $F$ satisfies Assumption A with respect to $dV_X$, then $\widetilde F$
satisfies Assumption A with respect to $\beta^n$. Replacing $\widetilde F$ by
$F$, we shall henceforth work with
$$
(dd^c u)^n=F(u,x)\,\beta^n .
$$
All $L^p$-spaces below are taken with respect to $\beta^n$; in particular,
$L^p(\Omega):=L^p(\Omega,\beta^n)$. The form $\beta$ is used only to define a
smooth positive volume form; it need not be closed.

\begin{setting}\label{setting}
Let $X$ be a Stein space of dimension $n\ge 1$, reduced and locally irreducible. Fix a smooth positive $(1,1)$-form $\beta$ on $X$.  Let $\Omega\Subset X$ be a relatively compact strongly pseudoconvex domain admitting a $C^2$-smooth exhaustion function $\rho$. Let $\phi\in C^0(\partial\Omega)$.
\end{setting}

\begin{definition}
    A plurisubharmonic function $v\in \psh^{\infty}(\Omega)$ is a subsolution to the Dirichlet problem \eqref{main-eq-general} with data $(\phi, F)$ if the following two conditions are satisfied:
    
\begin{enumerate}
        \item $v^*(\zeta):= \limsup_{\Omega\ni z\to \zeta} v(z) \le \phi(\zeta)$, for any $\zeta  \in \partial \Omega$ ,
        \item   $(\ddc v)^n  \ge F(v,x)\beta^n$ weakly in $\Omega$.
    \end{enumerate}
\end{definition}

We consider the family $S_{\phi, F}$ of all subsolutions to the Dirichlet problem \eqref{main-eq-general} with data $(\phi,F)$, together with its upper envelope whenever the latter exists.

\subsection{The subsolution property}

\begin{lemma}\label{lem subsolution}
    In Setting \ref{setting}, assume that the Dirichlet problem \eqref{main-eq-general} with data  $(\phi, F)$ admits a subsolution $v_0 \in S_{\phi, F}(\Omega)$. Then the upper envelope of subsolutions 
    \begin{equation}
        U:=U_{\phi, F} :=\sup \{ v; v\in S_{\phi, F}(\Omega)\}
    \end{equation}
is a subsolution to the Dirichlet problem \eqref{main-eq-general}, i.e. $U_{\phi,F} \in  S_{\phi, F}(\Omega)$. Moreover, if $v_0=\phi$ on $\partial \Omega$, then 
\begin{equation}
    \lim_{z\rightarrow \zeta} U(z)=\phi(\zeta), \quad \text{for any } \zeta \in \partial \Omega .
\end{equation}    
\end{lemma}

\begin{proof}

\bigskip
Step 1: For any $u \in S_{\phi, F}(\Omega)$, since $u$ is plurisubharmonic, we have
$$ u\le M_{\phi}:=\max_{\partial\Omega}\phi. $$
Moreover, $u\ge v_0$. It follows that the envelope $U$ is a well-defined bounded function on $\Omega$, and that its upper semicontinuous regularization $U^*$ is a bounded plurisubharmonic function on $\Omega$  by Corollary \ref{cor of family}.

\bigskip
Step 2: We claim that
$
(\ddc U^*)^n \ge F( U^*,x)\beta^n
$
in the weak sense on $\Omega$.

By Choquet's lemma (see \cite[Lemma 4.31]{GuedjZeriahi2017}), there exists a countable subfamily $(u_j)\subset S_{\phi, F}(\Omega)$ such that 
$U^*=\left(\sup_j u_j\right)^*$ on $\Omega$.
For  $u_1,u_2 \in  S_{\phi, F}(\Omega)$, we set $w:=\max\{ u_1,u_2\}$. Then $w\in \psh^{\infty}(\Omega)$ and $w^*(\zeta)\le \phi(\zeta) $ for any $\zeta \in \partial \Omega$. It follows from the local maximum principle \cite[Corollary 3.28]{GuedjZeriahi2017} that
\begin{equation}
    (\ddc w)^n \ge 1_{\{u_1\ge u_2\}}(\ddc u_1)^n +1_{\{u_1 < u_2\}}(\ddc u_2)^n.
\end{equation}
Since  $u_1,u_2 \in  S_{\phi, F}(\Omega)$, we get $(\ddc u_1)^n \ge F(u_1,x)\beta^n$ and $(\ddc u_2)^n \ge F(u_2,x)\beta^n$.
Therefore, 
\begin{equation}
     (\ddc w)^n \ge  1_{\{u_1\ge u_2\}} F(u_1,x) \beta^n + 1_{\{u_1 < u_2\}}F(u_2,x)\beta^n = F(w,x)\beta^n.
\end{equation}
Hence, $w \in S_{\phi, F}(\Omega)$. By induction, any finite maximum of elements of $ S_{\phi, F}(\Omega)$ belongs to $ S_{\phi, F}(\Omega)$. For each $j\in \mathbb{N}$, let $v_j := \max \{ u_i; 1\le i \le j \}$. Then $(v_j)$ is a nondecreasing sequence in $S_{\phi, F}(\Omega)$, and $U^*=(\sup_j v_j)^*$. Set $v= \sup_j v_j$. 
By the Bedford--Taylor monotone convergence theorem,
\begin{equation}\label{LHS-LIMIT}
(\ddc v_j)^n \longrightarrow (\ddc v^*)^n=(\ddc U^*)^n
\qquad \text{weakly in } \Omega.
\end{equation}
Since each $v_j \in  S_{\phi, F}(\Omega)$, we have $ F(v_j,x)\le F(M_\phi,x)=:g_\phi(x)\in L^p(\Omega)$ for some $p>1$ and $(\ddc v_j)^n \ge F(v_j,x)\beta^n $ weakly in $\Omega$.
Since $\Omega\Subset X$, the measure $\beta^n$ is finite on $\Omega$. Since $v^*$ is the upper semicontinuous regularization of $v$, it follows from \cite[Proposition 1.40]{GuedjZeriahi2017} that $v=v^*$ a.e. on $\Omega$.
On the other hand, since $v_j\uparrow v$ pointwise on $\Omega$ and, for each fixed $x\in \Omega$, the map $t\longmapsto F(t,x)$
is continuous, we obtain $F(v_j,x)\to F(v,x)$ for every $x\in \Omega$.
Hence
\begin{equation}\label{RHS-LIMIT}
F(v_j,x)\longrightarrow F(v^*,x)
\quad \text{a.e. on } \Omega.
\end{equation}
 Therefore, by the dominated convergence theorem,
$F(v_j,\cdot)\rightarrow F(v^*,\cdot)$ in $ L^1(\Omega)$.
In particular, we obtain
\begin{equation}\label{limit-1}
F(v_j,x)\,\beta^n \longrightarrow F(v^*,x)\,\beta^n
\qquad \text{weakly as measures on } \Omega.
\end{equation}
Passing to the limit in \eqref{limit-1}, and using \eqref{LHS-LIMIT} and \eqref{RHS-LIMIT}, we obtain
$$
(\ddc U^*)^n \ge F(U^*,x)\beta^n
\qquad \text{weakly in } \Omega.
$$

\bigskip
Step 3: 
Let $U_{\phi} := U_{\phi,0}$ be the maximal plurisubharmonic function on $\Omega$ with boundary values $\le \phi$. Then, by \cite[Lemma 3]{GuedjGuenanciaZeriahi2023}, the envelope $U_{\phi}$ exists, satisfies $U_{\phi}=U_{\phi}^{*}$, and $\lim_{z\rightarrow \zeta} U_{\phi}(z)=\phi(z)$ for any $\zeta \in \partial \Omega$.
Since we have 
\begin{equation}\label{1-1}
    v_0 \le U \le U_{\phi} \quad \text{ on } \Omega
\end{equation}
and $U_{\phi}$ is upper semicontinuous, it follows 
$v_0 \le U^* \le U_{\phi}$ on $\Omega$. Then we have $U^*(\zeta)\le \phi(\zeta)$ for any $\zeta\in \partial \Omega$. It follows that $ U^*\in S_{\phi, F}(\Omega)$. Hence $U^* \le U$ and finally $U^*=U$ on $\Omega$. The boundary condition follows from \eqref{1-1}.
\end{proof}

\subsection{Stability estimate}---
 We record a stability estimate that will be used repeatedly below. It is a direct consequence of \cite[Proposition 1.8]{GuedjGuenanciaZeriahi2023}.

\begin{prop}\label{prop:stability}
Let $\Omega \Subset X$ be an open set, let $\beta$ be a smooth positive $(1,1)$-form on $X$, and let $u,v\in \psh^\infty(\Omega)$. Assume that
$$
(\ddc u)^n = F(u,x)\,\beta^n
\qquad \text{on } \Omega,
$$
where $F$ satisfies Assumption~\ref{ass:F}. Assume moreover that there exists a constant $M\in\mathbb{R}$ such that $u \le M$ on $\Omega$ and set
$g_M(x):=F(M,x)$.
If $g_M\in L^p(\Omega)$, then for any $0<\gamma<\frac{1}{nq+1}$, there exists a constant
$C=C\bigl(\gamma,\|g_M\|_{L^p(\Omega)}\bigr)>0$
such that
$$
\sup_{\Omega}(v-u)_+
\le
\sup_{\partial\Omega}(v-u)_+^*
+
C\|(v-u)_+\|_{L^1(\Omega)}^\gamma,
$$
where $q$ is the conjugate exponent of $p$, namely $\frac1p+\frac1q=1$.
\end{prop}

\begin{proof}
Define $f(x):=F(u(x),x)$.
Since $F$ is measurable in $x$ and continuous in the first variable, the function $f$ is measurable on $\Omega$. Moreover, since the map $t\mapsto F(t,x)$ is nondecreasing by Assumption~\ref{ass:F} and $u\le M$ on $\Omega$, we obtain for $x\in \Omega$
$$
f(x)=F(u(x),x)\le F(M,x)=g_M(x).
$$
Since $g_M\in L^p(\Omega)$, it follows that
$f\in L^p(\Omega)$ and 
$\|f\|_{L^p(\Omega)}
\le
\|g_M\|_{L^p(\Omega)}.$

On the other hand, the equation satisfied by $u$ can be rewritten as
$(\ddc u)^n=f\,\beta^n$ on $\Omega$.
Therefore, we can apply the stability estimate in \cite[Proposition 1.8]{GuedjGuenanciaZeriahi2023} for bounded plurisubharmonic functions whose Monge--Amp\`ere measure has an $L^p$ density. It yields that for every
$0<\gamma<\frac{1}{nq+1}$,
there exists a constant
$
C_0=C_0\bigl(\gamma,\|f\|_{L^p(\Omega)}\bigr)>0
$
such that
$$
\sup_{\Omega}(v-u)_+
\le
\sup_{\partial\Omega}(v-u)_+^*
+
C_0\|(v-u)_+\|_{L^1(\Omega)}^\gamma.
$$
Since $\|f\|_{L^p(\Omega)}\le \|g_M\|_{L^p(\Omega)}$, after increasing $C$ we write
$
C=C\bigl(\gamma,\|g_M\|_{L^p(\Omega)}\bigr).
$
\end{proof}

\subsection{A preliminary existence result via the balayage process}---
Combining \cite[Corollary 1.2]{Kolodziej2000} with \cite[Theorem A]{GuedjGuenanciaZeriahi2023}, we establish the existence of a solution in $\psh^{\infty}(\Omega)$ to the Dirichlet problem \eqref{main-eq-general}. The argument also serves to illustrate the balayage process, which will be used later in the proof of Theorem \ref{real main thm 1}.

\begin{lemma}\label{lem 2}
Let $\Omega \Subset X$ be a strongly pseudoconvex domain, and let $F$ satisfy Assumption~\ref{ass:F}. 
For $\phi \in C^0(\partial \Omega)$, there exists a unique function $w\in \psh^\infty (\Omega)$ such that 
    \begin{equation}\label{2-1}
        (\ddc w)^n = F(w,x)\beta^n \text{ on } \Omega, \quad \text{with } w|_{\partial \Omega}=\phi.
    \end{equation}
\end{lemma}

\begin{proof}
By Assumption~\ref{ass:F}, there exists a function $g\in L^p(\Omega, \beta^n)$ for some $p>1$ such that
    $$
    0\le F(t,x)\le g(x)
    \qquad \text{for all } (t,x)\in ]-\infty,\max_{\partial\Omega}\phi]\times \Omega.
    $$
We first construct a subsolution to equation \eqref{2-1}.
By \cite[Theorem A]{GuedjGuenanciaZeriahi2023}, there exists a unique solution $\psi\in \psh(\Omega)\cap C^0(\overline{\Omega})$ to the Dirichlet problem $(\ddc \psi)^n = g\beta^n$ in $\Omega$, $\psi|_{\partial\Omega} = \phi$ on $\partial\Omega$. 
By the maximum principle, we have $\psi\le \max_{\partial\Omega}\phi$ on $\Omega$.
Then  we obtain
$$
(\ddc \psi)^n = g\,\beta^n
\ge F(\psi,x)\,\beta^n
\qquad \text{in } \Omega.
$$
Hence $\psi$ is a subsolution to the Dirichlet problem.

We consider the envelope of subsolutions $w:=\sup \{ v; v\in S\}$, where 
\begin{equation}
    S:= \{ v; v \in \psh^{\infty}(\Omega), (\ddc v)^n \ge F(v,x)\beta^n \text{ and } v|_{\partial \Omega} \le \phi \}.
\end{equation}
By the maximum principle, we have $w\le \max_{\partial\Omega}\phi$, while $w\ge \psi$ since $\psi \in S$. The Lemma \ref{lem subsolution} guarantees that $w$ is again a subsolution to the Dirichlet problem \eqref{2-1}
i.e. $(\ddc w)^n \ge F(w,x)\beta^n$ in $\Omega$ and $w=\phi$ on $\partial\Omega$.

It remains to show that $(\ddc w)^n = F(w,x)\beta^n$ in $\Omega_{\mathrm{reg}}$.
Fix a point $a\in \Omega_{\mathrm{reg}}$. Since $\Omega_{\mathrm{reg}}$ is a complex manifold near $a$, we can choose a relatively compact neighborhood $B\Subset \Omega_{\mathrm{reg}}$ of $a$ such that there exists a biholomorphism
$
T: V \to \mathbb{V}
$
from an open neighborhood $V$ of $\overline{B}$ onto an open neighborhood $\mathbb{V}$ of the closed Euclidean unit ball $\overline{\mathbb{B}}$ satisfying
$T(B)=\mathbb{B}.$
Since $w\in \psh^\infty(\Omega)$, the restriction $w|_{\partial B}$ is an upper semicontinuous function on the compact metric space $\partial B$, and is therefore bounded from above. Then there exists a decreasing sequence $(h_j)_{j\in \mathbb{N}}\subset C(\partial B)$ such that $h_j \downarrow w|_{\partial B}$ pointwise on $\partial B$(see  \cite[Exercise 1.1]{GuedjZeriahi2017}). For each $j\in \mathbb{N}$, it follows again from \cite[Theorem A]{GuedjGuenanciaZeriahi2023} that there exists a unique solution $\psi_j\in \psh(B)\cap C^0(\overline{B})$ to the Dirichlet problem:
\begin{equation*}
\left\{
\begin{aligned}
(\ddc \psi_j)^n &= \beta^n \qquad &&\text{in } B,\\
\psi_j|_{\partial B} &= h_j \qquad &&\text{on } \partial B.
\end{aligned}
\right.
\end{equation*}
Hence, for each $j\in \mathbb{N}$, the boundary datum
$h_j=\psi_j|_{\partial B}$
is admissible for \cite[Corollary 1.2]{Kolodziej2000}, since $\psi_j$ is plurisubharmonic on $B$ and continuous on $\overline{B}$. Then we consider the Dirichlet problem for $u_j\in \psh^{\infty}(B)$:
\begin{equation}\label{2-2}
\left\{
\begin{aligned}
 (\ddc u_j)^n& = F(u_j,x)\beta^n  &\text{in } \; B, \\
  \lim_{z\to\zeta}u_j(z)&=h_j(\zeta) &\text{for } \zeta \in \partial B.
\end{aligned}
\right.
\end{equation}
Using a result in \cite[Corollary 1.2]{Kolodziej2000}, we can find a solution $u_j \in \psh^{\infty}(B)$ to this problem. By the comparison principle Proposition \ref{comparison}, $w\le u_{j+1} \le u_j$ on $B$ for every $j$. Thus $(u_j)$ is a decreasing sequence in $\psh^\infty(B)$, bounded from below by function $\psi$. We set $u:=\lim_{j\to\infty} u_j$. Then $u\in \psh^{\infty}(B)$ and $\limsup_{z \to \zeta} u(z) \le w(\zeta)$ for $\zeta \in \partial B$. By the continuity of the complex Monge--Amp\`ere operator along decreasing plurisubharmonic sequences, 
$$
(\ddc u_j)^n \longrightarrow (\ddc u)^n
\quad \text{weakly on } B.
$$
Since $u_j\downarrow u$ pointwise on $B$ and, for each fixed $x\in B$, the map $t\longmapsto F(t,x)$ is continuous, we obtain $F(u_j,x)\to F(u,x)$ pointwise on $B$. Since $B\Subset \Omega$ and $p>1$, we have
$g\in L^p(B,\beta^n)\subset L^1(B,\beta^n)$. The dominated convergence theorem gives 
$$
F(u_j,\cdot)\longrightarrow F(u,\cdot) \quad \text{in } L^1(B,\beta^n). 
$$
Passing to the limit in $(\ddc u_j)^n = F(u_j,x)\beta^n$, we obtain
$(\ddc u)^n = F(u,x)\beta^n$ on $B$. Since $u_j \downarrow u$ and $w\le u_j$ on $B$ for every $j$, we obtain $w \le u$ on $B$.

We define a function $\widetilde{u}=u$ on $B$ and $\widetilde{u}=w$ on $\Omega\setminus  B$.
Then $\widetilde{u}$ is a subsolution to the Dirichlet problem \eqref{2-1}. Therefore $\widetilde{u}\le w$ on $\Omega$. This yields $\widetilde{u}=u=w$ on $B$. It follows that $(\ddc w)^n= F(w,x)\beta^n$ on $B$. Since $B\subset \Omega$ was arbitrary, we conclude that $(\ddc w)^n= F(w,x)\beta^n$ on $\Omega_{\mathrm{reg}}$. Hence,  $(\ddc w)^n= F(w,x)\beta^n$ on $\Omega$, because both measures vanish on $\Omega_{\mathrm{sing}}$. The uniqueness follows directly from comparison Lemma \ref{0-1}.
\end{proof}

\subsection{Existence of a solution}---
We now complete the proof of the main theorem. 

\begin{theorem}\label{real main thm 1}
    In Setting \ref{setting}, we have $S_{\phi, F}(\Omega)\neq \varnothing$ and its upper envelope 
\begin{equation}
    u:= U:=\sup \{ v; v\in S_{\phi, F}(\Omega)\}
\end{equation}
satisfies $u\in \psh(\Omega) \cap  C(\Omega)$. Moreover, $u$ extends continuously to $\overline{\Omega}$ and is the unique solution to the Dirichlet problem \eqref{main-eq-general}.
    
\end{theorem}

\begin{proof}
We proceed in several steps.

\bigskip
Step 1: Assume first that $\phi \in C^2(\partial\Omega)$. Indeed, we choose an extension of the boundary datum to a neighborhood of $\overline{\Omega}$, still denoted by $\phi$, such that
$\max_{\overline{\Omega}}\phi=\max_{\partial\Omega}\phi.$
We then set
$
M_\phi:=\max_{\partial\Omega}\phi=\max_{\overline{\Omega}}\phi
$
and we also assume 
$
g_\phi(x):=F(M_\phi,x)\in L^{\infty}(\Omega).$ Up to replacing $\rho$ by a positive multiple of itself, we may assume that $\ddc \rho \ge \beta$. For $A > (2 ||g_{\phi}||_{\infty})^{\frac{1}{n}}$ large enough depending on $\|\phi\|_{C^2(\overline{\Omega})}$, the function $v_{\phi}=v_{\phi}^A:=A\rho+\phi$ is plurisubharmonic and smooth in a neighborhood of $\overline{\Omega}$ and 
\begin{equation}
    (\ddc v_{\phi})^n\ge \frac{1}{2}A^n\beta^n \ge g_{\phi}\beta^n =F(M_{\phi},x)\beta^n  \ge F(v_{\phi},x)\beta^n.
\end{equation}
Since $v_{\phi}=\phi$ on $\partial \Omega$, we have $v_{\phi} \in S_{\phi,F}$.
By Lemma \ref{lem subsolution}, the upper envelope $U$ of subsolutions is well defined and satisfies $(\ddc U)^n \ge F(U,x)\beta^n$ on $\Omega$ together with
$ \lim_{z\rightarrow \zeta} U(z) =\phi(z)$ for any  $\zeta \in \partial \Omega.$
The identity $(\ddc U)^n =F(U,x)\beta^n $ in $\Omega$ can be obtained by a balayage argument exactly as in the proof of Lemma \ref{lem 2}.

Let $\Omega'$ be a neighborhood of $\overline{\Omega}$ such that $v_\phi$ is defined and plurisubharmonic on $\Omega'$. Define
$$
V=
\begin{cases}
U & \text{on } \Omega,\\
v_\phi & \text{on } \Omega'\setminus \Omega.
\end{cases}
$$
Then $V$ is plurisubharmonic on $\Omega'$.
We can always assume that $\Omega'$ is a Stein space. By Forn\ae ss-Narasimhan approximation Theorem \cite[Theorem 5.5]{FornaessNarasimhan1980}, there exists a decreasing sequence $(V_j)$ of smooth plurisubharmonic functions on $\Omega'$ which converges to $V$ pointwise on $\Omega'$. We set $v_j :=V_j$ on $\overline{\Omega}$. Then $v_j \in \psh(\Omega)\cap C^0(\overline{\Omega})$ and $v_j\downarrow U$ pointwise on $\Omega$. By the monotone convergence theorem, $\|v_j-U\|_{L^1(\Omega)}\to 0$.

It follows from stability estimate in Proposition \ref{prop:stability} for $(\ddc U)^n= F(U,x)\beta^n$ that 
\begin{align}\label{2-3}
   \sup_{\Omega} (v_j -U) \le \sup_{\partial \Omega}(v_j -U)+ C||v_j-U||^\gamma_{L^1(\Omega)},
\end{align}
where $0<\gamma<\frac{1}{nq+1}$ and $C=C(\gamma, \max_{\partial\Omega}\phi , \|g_{\phi}\|_{L^p(\Omega)})>0$. Since $(v_j)$ decreasing to $U=\phi$ on $\partial \Omega$ and $\phi \in C^0(\partial\Omega)$, Dini's lemma ensures that $\sup_{\partial\Omega}|v_j-\phi|\to 0$. It follows from \eqref{2-3} that $(v_j)$ uniformly converges to $U$ on $\overline{\Omega}$, hence $U$ is continuous on $\overline{\Omega}$.

\bigskip
Step 2: Assume $\phi \in C^0(\partial \Omega)$ and $g_{\phi}=F(\max_{\partial\Omega} \phi,x) \in L^{\infty}(\Omega)$. For $t>\max_{\partial\Omega} \phi$, we extend $F$ by setting $F(t,x)=F(\max_{\partial\Omega} \phi,x)$. Take a decreasing sequence of smooth functions $(\phi_j)_{j\in \mathbb{N}}$ converging to $\phi$ on $\partial \Omega$ so that $\phi_j -\frac{1}{j} \le \phi \le \phi_j$ on $\partial\Omega$ for any $j\ge 1$. It follows from Step 1 that $U_{\phi_j,F}\in \psh(\Omega) \cap C^0(\overline{\Omega})$.
Since $(\ddc U_{\phi,F})^n \ge F(U_{\phi,F},x)\beta^n$  and $(\ddc U_{\phi_j,F})^n = F(U_{\phi_j,F},x)\beta^n$, by comparison Lemma \ref{0-1}, we have 
$  U_{\phi,F} \le U_{\phi_j,F}$ on $\overline{\Omega}$.
Since  $$[\ddc (U_{\phi_j,F}-\frac{1}{j})]^n = (\ddc U_{\phi_j,F})^n =F(U_{\phi_j,F},x)\beta^n \ge F(U_{\phi_j,F}-\frac{1}{j},x)\beta^n$$
and $U_{\phi_j,F}-\frac{1}{j} \le \phi$ on $\partial \Omega$, we obtain $(U_{\phi_j,F}-\frac{1}{j})\in S_{\phi, F}(\Omega) $ . It follows that $ U_{\phi_j,F}-\frac{1}{j} \le   U_{\phi,F}$ on $ \overline{\Omega}$. Hence,
\begin{equation}
  0\le  U_{\phi_j,F}- U_{\phi,F} \le \frac{1}{j}  , \quad \text{on } \overline{\Omega}.
\end{equation}
Then the nonincreasing sequence $(U_{\phi_j,F})$ uniformly converges to $U_{\phi, F}$ on $\overline{\Omega}$.
Hence, $U=U_{\phi,F}$ is continuous on $\overline{\Omega}$ and $U=\phi$ on $\partial \Omega$. By the continuity of the complex Monge--Amp\`ere operator along monotone sequences and the continuity of $F$ in its first variable, passing to the limit in $(\ddc U_{\phi_j,F})^n=F(U_{\phi_j,F},x)\beta^n$, we obtain $(\ddc U)^n=F(U,x)\beta^n$ on $\Omega$.

\bigskip
Step 3: Assume $\phi \in C^0(\partial \Omega)$ and $g_{\phi}=F(\max_{\partial\Omega} \phi,x) \in L^{p}(\Omega)$, with $p>1$. Let $M:=\max_{\partial\Omega}\phi$. For each $j\in \mathbb{N}$, define
$$
F_j(t,x):=\min\{F(t,x),j\},
\qquad (t,x)\in ]-\infty,M]\times\Omega.
$$
Then each $F_j$ satisfies Assumption~\ref{ass:F}, and moreover
$$
0\le F_j(t,\cdot) \le F_j(M,\cdot)=\min\{g_{\phi},j\}\in L^{p}(\Omega).
$$
The sequence $(F_j)$ increases pointwise to $F$ as $j\to +\infty$ on $]-\infty,M]\times\Omega$. Set
$U_j:=U_{\phi,F_j}$, $j\in\mathbb{N}$.
By comparison Lemma~\ref{0-1}, the sequence $(U_j)_{j\in\mathbb{N}}$ is decreasing, and each
$$
U_j:=U_{\phi,F_j}\in \psh(\Omega)\cap C^0(\overline{\Omega}).
$$
We first show that the functions $U_j$ are uniformly bounded. Let $v:=U_{\phi,0}$ denote the maximal plurisubharmonic function on $\Omega$ with boundary values  $\le \phi$. We recall that $v$ exists, satisfies $v=v^*$, and has the standard properties of the envelope constructed in \cite[Lemma 3]{GuedjGuenanciaZeriahi2023}.
Since $0\le F_j$,  comparison Lemma~\ref{0-1} yields $U_j\le v$ on  $\Omega$
for all $j\in\mathbb{N}$. On the other hand, by stability estimate in Proposition~\ref{prop:stability}, for every $j\in\mathbb{N}$,
$$
\|v-U_j\|_{L^{\infty}(\Omega)}
\le
C\,\|v-U_j\|_{L^1(\Omega)}^{\gamma},
$$
where $0<\gamma<\frac{1}{nq+1}$ and $C>0$ depends only on $\gamma$, $M$, and $\|g_{\phi}\|_{L^p(\Omega)}$. Since
$$
\|v-U_j\|_{L^1(\Omega)}
\le
\mathrm{Vol }\,(\Omega)\,\|v-U_j\|_{L^{\infty}(\Omega)},
$$
we obtain $\|v-U_j\|_{L^{\infty}(\Omega)}^{1-\gamma}\le C'$, 
where $C'>0$ is a uniform constant. Hence, 
$$\sup_{j}\|U_j\|_{L^{\infty}(\Omega)}<+\infty.
$$
Since $(U_j)$ is a decreasing sequence of continuous functions on $\overline{\Omega}$, it converges pointwise on $\overline{\Omega}$ to a bounded function $V$. Moreover,
$V|_\Omega$ is plurisubharmonic and $V|_{\partial \Omega} = \phi $.
Since $U_j$ decreasing to $V$, the continuity property of the complex  Monge--Amp\`ere operator along monotone sequences yields 
\begin{equation}
    (\ddc U_j)^n \to (\ddc V)^n \qquad \text{weakly on } \Omega.
\end{equation}

We now pass to the limit in the right-hand side. Since $U_j(x)\downarrow V(x)$ for every $x\in\Omega$ and, for each fixed $x\in\Omega$, the map $t\longmapsto F(t,x)$
is continuous, we have
$$
F(U_j(x),x)\longrightarrow F(V(x),x)
\qquad \text{for every } x\in\Omega.
$$
Hence, $F_j(U_j(x),x)=\min\{F(U_j(x),x),j\}\longrightarrow F(V(x),x)$ for every $x\in\Omega$.
Moreover, since $U_j\le v\le M$ on $\Omega$ and $t\mapsto F(t,x)$ is nondecreasing, we have
$$
0\le F_j(U_j,x)\le F(U_j,x)\le F(M,x)=g_{\phi}(x)
\qquad \text{on } \Omega.
$$
Because $g_{\phi}\in L^p(\Omega)$ with $p>1$, it follows that $g_{\phi}\in L^1(\Omega)$.
Therefore, by the dominated convergence theorem,
$F_j(U_j,\cdot)\longrightarrow F(V,\cdot)$ in $L^1(\Omega)$.
In particular, we obtain
$$
F_j(U_j,x)\,\beta^n\longrightarrow F(V,x)\,\beta^n
\qquad \text{weakly as measures on } \Omega.
$$
Passing to the limit in $(\ddc U_j)^n=F_j(U_j,x)\beta^n$,
we obtain
$$
(\ddc V)^n=F(V,x)\beta^n
\qquad \text{on } \Omega.
$$

Since $V\in S_{\phi,F}(\Omega)$, by the definition of the upper envelope $U$, we have
$V\le U$ on $\Omega$.
On the other hand, since $F_j \le F$, the function $U$ is a subsolution in $S_{\phi,F_j}(\Omega)$. Then we have $U\le U_j$ for every $j$. Passing to the limit as $j\rightarrow +\infty$, we obtain $U\le V$. Hence, $U=V$ on $\Omega$
and consequently
$(\ddc U)^n=F(U,x)\beta^n$ on $ \Omega$.

Since $U_j\downarrow U$ on $\Omega$ and the sequence $(U_j)$ is uniformly bounded, the dominated convergence theorem gives $\|U_j-U\|_{L^1(\Omega)}\longrightarrow 0$.
By stability estimate in Proposition~\ref{prop:stability}, there exists a constant $C>0$, depending only on $\gamma$, $M$, and $\|g_{\phi}\|_{L^p(\Omega)}$, such that
$$
\sup_{\Omega}(U_j-U)
\le
C\,\|U_j-U\|_{L^1(\Omega)}^{\gamma}
$$
for all $j\in\mathbb{N}$. It follows that $U_j\to U$ uniformly on $\Omega$. Since each $U_j$ is continuous on $\overline{\Omega}$ and $U_j|_{\partial\Omega}=\phi$,
we conclude that $U$ extends continuously to $\overline{\Omega}$ and satisfies $U|_{\partial\Omega}=\phi$. The uniqueness follows directly from comparison Lemma \ref{0-1}.
\end{proof}

\section{Regularity near the boundary}

For the rest of the paper, we continue to work under the assumptions of Setting~\ref{setting}. Let $u$ be the solution of \eqref{main-eq-general} obtained in Theorem~\ref{main thm}. Set $ M:=\max_{\partial\Omega}\phi. $ The maximum principle yields $u\le M$ on $\Omega$. Define
$
f(x):=F(u(x),x)
$
for $x\in\Omega$. By Assumption~\ref{ass:F}, we obtain $0\le f\le g_\phi\in L^p(\Omega)$ on  $\Omega$, and hence $f\in L^p(\Omega)$. Consequently, \eqref{main-eq-general} is reduced to the fixed-density Dirichlet problem
$$
MA(\Omega,\phi,f):
\left\{
\begin{aligned}
(\ddc u)^n &= f\,\beta^n
\qquad &&\text{in } \Omega,\\
u&=\phi
\qquad &&\text{on } \partial\Omega.
\end{aligned}
\right.
$$
For later use, we write $u(\Omega,\psi,h)$ for the unique solution of the corresponding fixed-density problem with boundary value $\psi$ and density $h$ with respect to $\beta^n$.

\subsection{Estimates near the regular boundary}---
We first recall the following result from \cite[Lemma 2.2]{CerqueiraGoncalves2026}. The idea comes from \cite[Lemma 2.2]{GuedjKolodziejZeriahi2008}.
\begin{lemma}
\label{lem-zero-boundary-barrier}
Let $\Omega\Subset X$ be as in Setting~\ref{setting}. Let
$0\le \hat{f}\in L^p(\Omega)$ for some $p>1$, and $\hat{f}$ is bounded near $\partial \Omega$.
There exists a barrier 
$b_{\hat{f}}\in \psh(\Omega)\cap C^0(\overline{\Omega})$ such that:
\begin{enumerate}
    \item $b_{\hat{f}}(\xi)=0, \forall \xi \in \partial \Omega$;
    \item  $b_{\hat{f}}\le u(\Omega,0,\hat{f}) \le 0$ and $(\ddc b_{\hat{f}})^n \ge \hat{f}\beta^n$ in $\Omega$;
    \item there exists a constant $C>0$ such that, for all $x,y\in \overline{\Omega}$, one has
$|b_{\hat{f}}(x)-b_{\hat{f}}(y)|
\le C\, d_\beta(x,y) .$
\end{enumerate}
\end{lemma}

\begin{lemma}\label{lem-homogeneous-boundary-barrier}
Let $\Omega\Subset X$ be as in Setting~\ref{setting}, and let
$\phi\in C^0(\partial\Omega)$. There exist barriers 
$h_\phi,h_{-\phi}\in \psh(\Omega)\cap C^0(\overline{\Omega})$ such that 
\begin{enumerate}
    \item $h_{\phi}(\xi)=\phi(\xi)=-h_{-\phi}(\xi), \forall \xi \in \partial \Omega$,
    \item     Set $u_\phi:=u(\Omega,\phi,0)$. Then $h_{\phi}\le u_{\phi} \le -h_{-\phi}$ in $\Omega$,
\item There exists a constant $C>0$ such that for all $x,y\in \overline{\Omega}$, one has
$$
|h_\phi(x)-h_\phi(y)|+
|h_{-\phi}(x)-h_{-\phi}(y)|
\le C\,\omega_\phi\!\left(d_\beta(x,y)^{1/2}\right),
$$
where
$\omega_\phi(t):=
\sup\left\{
|\phi(\xi)-\phi(\eta)|\ ;\
\xi,\eta\in\partial\Omega,\ d_\beta(\xi,\eta)\le t
\right\}.$
\end{enumerate}
\end{lemma}

\begin{proof}
The construction follows the boundary barrier argument in
\cite[Lemma~2.3]{CerqueiraGoncalves2026}, whose idea goes back to
\cite[Proposition~4.4]{Charabati2015}. The only modification is that a boundary
point $\xi\in\partial\Omega$ may belong to $X_{\rm sing}$. Thus the local
barrier is constructed in an ambient embedding
$\iota_\xi:U_\xi\hookrightarrow\mathbb C^{N_\xi}$, using the strictly
plurisubharmonic defining function $\rho$ and the squared ambient distance
$|\iota_\xi(x)-\iota_\xi(\xi)|^2$, rather than in a smooth coordinate chart of
$X_{\rm reg}$. Since, in the local embedding, the distance $d_\beta$ is locally comparable to
the ambient Euclidean distance, the same argument as in the smooth case gives
the estimate $C\omega_\phi(t^{1/2})$. By compactness of $\partial\Omega$,
the constants can be chosen uniformly after passing to a finite cover. Taking the
supremum of the local lower barriers gives $h_\phi$, and applying the same
construction to $-\phi$ gives $h_{-\phi}$. The stated inequalities and modulus
estimate follow.
\end{proof}

We now introduce a quantity which measures the local $\alpha$-H\"older control of the boundary datum $\phi$. Let
$
S:=X_{\rm sing}\cap\partial\Omega.
$
For $\lambda>0$, we set
$
\Gamma_\lambda
:=
\{\xi\in\partial\Omega:\ d_\beta(\xi,S)\ge \lambda\}.
$
Define
\begin{equation}\label{mathcal_H_phi}
\mathcal H_\phi(\lambda)
:=
1+
\sup_{\substack{\xi\in \Gamma_\lambda ,\eta\in\partial\Omega\\ \xi\ne\eta}}
\frac{|\phi(\xi)-\phi(\eta)|}{d_\beta(\xi,\eta)^\alpha}.
\end{equation}
In the case $S=\varnothing$, we interpret $H_\phi(\lambda)$ as the usual global
$\alpha$-Hölder constant of $\phi$ on $\partial\Omega$.
\begin{lemma}\label{lem:local-holder-boundary-constant}
Assume that $\phi$ is locally $\alpha$-H\"older continuous on
$\partial\Omega\cap X_{\rm reg}$. Then, for every compact set
$\Gamma\Subset \partial\Omega\cap X_{\rm reg}$, one has
$
\mathcal H_\phi(d_\beta(\Gamma,S))<+\infty .
$
\end{lemma}

\begin{proof}
Set
$\lambda_\Gamma:=d_\beta(\Gamma,S)>0.$
Since $\Gamma_{\lambda_\Gamma/2}\Subset\partial\Omega\cap X_{\rm reg}$, the local
$\alpha$-H\"older continuity of $\phi$ gives a constant $C>0$ such that
$
|\phi(\xi)-\phi(\eta)|
\le C d_\beta(\xi,\eta)^\alpha
$
for all $\xi,\eta\in \Gamma_{\lambda_\Gamma/2}$. In particular, if
$\xi\in\Gamma_{\lambda_\Gamma}$ and
$d_\beta(\xi,\eta)<\frac{\lambda_\Gamma}{2}$, then
$\eta\in\Gamma_{\lambda_\Gamma/2}$, and hence the above estimate applies.
It remains to consider the case
$d_\beta(\xi,\eta)\ge \frac{\lambda_\Gamma}{2}$. Since $\phi$ is bounded on
$\partial\Omega$, we have
$$
|\phi(\xi)-\phi(\eta)|
\le 2\|\phi\|_{C^0(\partial\Omega)}
\le
2\|\phi\|_{C^0(\partial\Omega)}
\left(\frac{2}{\lambda_\Gamma}\right)^\alpha
d_\beta(\xi,\eta)^\alpha .
$$
Combining the two estimates, after increasing the constant $C>0$, we get
$
|\phi(\xi)-\phi(\eta)|
\le C d_\beta(\xi,\eta)^\alpha
$
for all $\xi\in\Gamma_{\lambda_\Gamma}$ and all $\eta\in\partial\Omega$ with
$\xi\ne\eta$. Therefore,
$\mathcal H_\phi(\lambda_\Gamma)<+\infty $.
\end{proof}

\begin{lemma}
\label{lem-local-holder-boundary-barrier}
Let $\Omega\Subset X$ be as in Setting~\ref{setting}, and let
$\phi\in C^0(\partial\Omega)$. Assume that $\phi$ is locally
$\alpha$-H\"older continuous on $\partial\Omega\cap X_{\rm reg}$.
Then for any compact set $\Gamma\Subset \partial\Omega\cap X_{\mathrm{reg}}$, there exist two barriers $h_{\phi,\Gamma},h_{-\phi,\Gamma} \in \psh(\Omega)\cap C^0(\overline{\Omega})$ such that the following properties hold. \begin{enumerate} \item $h_{\phi,\Gamma}\le \phi \le -h_{-\phi,\Gamma}$ on $\partial\Omega$, and $h_{\phi,\Gamma}(\xi)=\phi(\xi)=-h_{-\phi,\Gamma}(\xi)$, $ \forall\,\xi\in \Gamma .$ \item Set $u_\phi:=u(\Omega,\phi,0)$. Then $ h_{\phi,\Gamma}\le u_\phi\le -h_{-\phi,\Gamma}$ in $\Omega$.

\item
There exists a constant $C>0$, independent of $\Gamma$, such that, for all
$x,y\in\overline\Omega$,
$$
|h_{\phi,\Gamma}(x)-h_{\phi,\Gamma}(y)|
+
|h_{-\phi,\Gamma}(x)-h_{-\phi,\Gamma}(y)|
\le
C\mathcal H_\phi(d_\beta(\Gamma,S))
d_\beta(x,y)^{\alpha/2}.
$$
\end{enumerate}
\end{lemma}

\begin{proof}
Set
$\lambda_\Gamma:=d_\beta(\Gamma,S)$ and $A:=\mathcal H_\phi(\lambda_\Gamma).$
Since $\Gamma\subset \Gamma_{\lambda_\Gamma}$, the definition of
$\mathcal H_\phi$ gives
$|\phi(\xi)-\phi(\eta)|
\le
A\,d_\beta(\xi,\eta)^\alpha,$ for $
\xi\in\Gamma$ and $  \eta\in\partial\Omega .
$

Define, for $\eta\in\partial\Omega$,
$\psi_+(\eta)
:=
\sup_{\xi\in\Gamma}
\left\{
\phi(\xi)-Ad_\beta(\xi,\eta)^\alpha
\right\}.$
 Let $\eta,\eta'\in\partial\Omega$.
For every $\xi\in\Gamma$, since $0<\alpha\le 1$, we have 
$d_\beta(\xi,\eta)^\alpha
\le
d_\beta(\xi,\eta')^\alpha+d_\beta(\eta,\eta')^\alpha.$
Therefore
$$
\phi(\xi)-A d_\beta(\xi,\eta)^\alpha
\ge
\phi(\xi)-A d_\beta(\xi,\eta')^\alpha
-A d_\beta(\eta,\eta')^\alpha .
$$
Taking the supremum over $\xi\in\Gamma$ yields
$\psi_+(\eta)
\ge
\psi_+(\eta')-A d_\beta(\eta,\eta')^\alpha .$
Interchanging $\eta$ and $\eta'$, we obtain
 $|\psi_+(\eta)-\psi_+(\eta')|
\le
A d_\beta(\eta,\eta')^\alpha .$
Hence,
$\psi_+\in C^{0,\alpha}(\partial\Omega).$
If $\eta\in\partial \Omega$, then for every $\xi\in\Gamma$,
$$
  \phi(\xi)-A  d_\beta(\xi,\eta)^\alpha=    \phi(\xi)-\mathcal H_\phi(d_\beta(\Gamma, S))d_\beta(\xi,\eta)^\alpha
\le \phi(\eta).
$$
Taking the supremum over $\xi\in\Gamma$ gives $\psi_+(\eta)\le\phi(\eta)$ for $\eta \in \partial \Omega$.
Thus $\psi_+\le\phi$ on $\partial\Omega$. Moreover, if $\eta\in\Gamma$, then
choosing $\xi=\eta$ gives $\psi_+(\eta)\ge\phi(\eta)$. Hence $\psi_+=\phi$ on
$\Gamma$.

Similarly, define
$\psi_-(\eta):=
\sup_{\xi\in\Gamma}
\left\{
-\phi(\xi)-A d_\beta(\xi,\eta)^\alpha
\right\}$, for $ \eta\in\partial\Omega .$
Then $\psi_-\in C^{0,\alpha}(\partial\Omega)$, 
$\psi_-\le -\phi$ on $\partial\Omega$ and 
$\psi_-=-\phi$ on $\Gamma$ .

We now apply the boundary barrier construction from Lemma~\ref{lem-homogeneous-boundary-barrier}
to the globally H\"older continuous boundary data $\psi_+$ and $\psi_-$.  Set
$h_{\phi,\Gamma}:=h_{\psi_+}$ and
$h_{-\phi,\Gamma}:=h_{\psi_-}$ from Lemma~\ref{lem-homogeneous-boundary-barrier}.
Since $\psi_+\le\phi$ on $\partial\Omega$, the comparison principle  gives
$$
h_{\phi,\Gamma}=h_{\psi_+}
\le u(\Omega,\psi_+,0)
\le u(\Omega,\phi,0)=u_\phi
\qquad\text{in } \Omega .
$$
Similarly,  we have
$
h_{-\phi,\Gamma}=h_{\psi_-}
\le u(\Omega,-\phi,0)$ in $\Omega$.
On the other hand, the maximum principle gives
$u(\Omega,\phi,0)+u(\Omega,-\phi,0)\le 0$
in $\Omega .
$
Consequently,
$$
u_\phi=u(\Omega,\phi,0)
\le
-u(\Omega,-\phi,0)
\le
-h_{-\phi,\Gamma}
\qquad\text{in } \Omega .
$$
Thus
$h_{\phi,\Gamma}\le u_\phi\le -h_{-\phi,\Gamma}$
in $\Omega .$
Finally, since $\psi_+=\phi$ and $\psi_-=-\phi$ on $\Gamma$, we obtain
$
h_{\phi,\Gamma}(\xi)=\phi(\xi)=-h_{-\phi,\Gamma}(\xi)$
 for $\xi\in\Gamma $.
Also, because $\psi_+\le\phi$ and $\psi_-\le-\phi$ on $\partial\Omega$, we have
$h_{\phi,\Gamma}\le\phi\le -h_{-\phi,\Gamma}$ on $ \partial\Omega.$
The modulus of continuity in Lemma~\ref{lem-homogeneous-boundary-barrier} for $h_{\psi_+}$ and $h_{\psi_-}$ yield, 
for all $x,y\in\overline{\Omega}$,
\begin{equation*}
|h_{\phi,\Gamma}(x)-h_{\phi,\Gamma}(y)|
+
|h_{-\phi,\Gamma}(x)-h_{-\phi,\Gamma}(y)|
\le
C \mathcal H_\phi(d_\beta(\Gamma, S)) d_\beta(x,y)^{\alpha/2}. \qedhere
\end{equation*}
\end{proof}

The following proposition is the main result of this section. It gives a local H\"older continuity for the solution $u$ near $\partial\Omega\cap X_{\rm reg}$.
\begin{prop}\label{near boundary K}
Let $u$ be the solution given by Theorem~\ref{main thm}. Assume that
$\phi$ is locally $\alpha$-H\"older continuous on
$\partial\Omega\cap X_{\rm reg}$. Let
$K\Subset \overline{\Omega}\cap X_{\rm reg}$
and assume that
$K\cap\partial\Omega\neq\varnothing$.
Then there exist constants $C_0>0$ and $\delta_{K}>0$ such that for every
$x,y\in K$ and every $0<s<\delta_{K}$ satisfying
$
d_\beta(x,\partial\Omega)<s, d_\beta(y,\partial\Omega)<s, d_\beta(x,y)<s,
$
one has
$$
|u(x)-u(y)|
\le
C_{0}  \mathcal H_\phi\left(\frac{1}{2}d_\beta(K\cap\partial\Omega,S)\right)s^{\alpha/2}.
$$
Moreover, for every $x\in K$ satisfying
$
d_\beta(x,\partial\Omega)<\delta_{K}
$,
there exists
$
\xi\in\Gamma_{\frac12 d_\beta(K\cap\partial\Omega,S)}
$
such that
$
d_\beta(x,\xi)=d_\beta(x,\partial\Omega)
$.
\end{prop}

\begin{proof}
Set
$r_K:=d_\beta(K\cap\partial\Omega,S)>0$.
We first claim that there exists $\lambda_K>0$ such that if
$x\in K$ and
$d_\beta(x,\partial\Omega)<\lambda_K$, 
then every nearest boundary point $\xi\in\partial\Omega$ satisfying
$d_\beta(x,\xi)=d_\beta(x,\partial\Omega)$
belongs to $\Gamma_{r_K/2}$.
Suppose not. Then there exist sequences $x_j\in K$ and
$\xi_j\in\partial\Omega\setminus\Gamma_{r_K/2}$ such that
$
d_\beta(x_j,\xi_j)
=
d_\beta(x_j,\partial\Omega)
\rightarrow 0
$. 
Passing to a subsequence, we may assume that
$x_j\rightarrow x_\infty\in K$.
Since $d_\beta(x_j,\xi_j)\to0$, we also have
$\xi_j\rightarrow x_\infty$. 
As $\partial\Omega$ is closed, it follows that
$x_\infty\in K\cap\partial\Omega$.
Hence, 
$d_\beta(x_\infty,S)\ge r_K.$
On the other hand, since $\xi_j\notin\Gamma_{r_K/2}$, we have
$d_\beta(\xi_j,S)<\frac12 r_K$.
Passing to the limit gives
$d_\beta(x_\infty,S)\le\frac12 r_K$.
This is a contradiction. Therefore the claim is proved.

We next obtain a boundary estimate for $u$ along $\Gamma_{r_K/2}$. Since $\Omega$ is strongly pseudoconvex,
after shrinking the neighborhood where the defining function $\rho$ is defined,
we may choose a strongly pseudoconvex domain $\widehat\Omega$ such that
$
\Omega\Subset \widehat\Omega\Subset X.
$
For instance, if $\Omega=\{\rho<0\}$, then one may take
$\widehat\Omega:=\{\rho-\varepsilon<0\}$
for $\varepsilon>0$ sufficiently small.

Extend $f$ by zero to $\widehat\Omega$ and denote the extension by $\widehat f$.
 Let $\widehat b$ be the lower barrier for
$u(\widehat\Omega,0,\widehat f)$ provided by
Lemma~\ref{lem-zero-boundary-barrier}. Thus
$
\widehat b\le u(\widehat\Omega,0,\widehat f)\le 0$ and $
(\ddc \widehat b)^n\ge \widehat f\,\beta^n$ in 
$ \widehat\Omega,$
and $\widehat b$ is Lipschitz near $\overline\Omega$.

Set
$\psi:=\phi-\widehat b|_{\partial\Omega}$.
Since $\widehat b$ is Lipschitz and $\phi$ is locally $\alpha$-H\"older
continuous on $\partial\Omega\cap X_{\rm reg}$, the function $\psi$ is locally
$\alpha$-H\"older continuous on $\partial\Omega\cap X_{\rm reg}$. Since $\widehat b$ is Lipschitz near $\overline\Omega$ and $\mathcal H_\phi\ge 1$,
we have, after increasing the constant,
$$
\mathcal H_\psi(r_K/2)\le C \mathcal H_\phi(r_K/2).$$
 Applying
Lemma~\ref{lem-local-holder-boundary-barrier} to $\psi$ and the compact set
$\Gamma_{r_K/2}$, we obtain a barrier
$
h_{\psi,\Gamma_{r_K/2}}\in\psh(\Omega)\cap C^0(\overline\Omega)
$
such that
$$
h_{\psi,\Gamma_{r_K/2}}\le\psi
\quad\text{on }\partial\Omega
\, \text{, and } \,
h_{\psi,\Gamma_{r_K/2}}(\xi)=\psi(\xi)
\quad
\text{ for } \xi\in\Gamma_{r_K/2}.
$$
And 
$|h_{\psi,\Gamma_{r_K/2}}(z)-h_{\psi,\Gamma_{r_K/2}}(\eta)|
\le
C \mathcal H_\phi\left(\frac{r_K}{2}\right) d_\beta(z,\eta)^{\alpha/2} $
for all $z,\eta\in\overline\Omega$.

Define
$
v_K:=\widehat b|_{\Omega}+h_{\psi,\Gamma_{r_K/2}}
$.
Then, 
$v_K\in\psh(\Omega)\cap C^0(\overline\Omega)$.
On $\partial\Omega$, we have
$
v_K
=
\widehat b+h_{\psi,\Gamma_{r_K/2}}
\le
\widehat b+\psi
=
\phi$.
For every $\xi\in\Gamma_{r_K/2}$, we also have
$v_K(\xi)=\phi(\xi).$
Since $h_{\psi,\Gamma_{r_K/2}}$ is plurisubharmonic, we get
$
(\ddc v_K)^n
\ge
(\ddc\widehat b)^n
\ge
f\beta^n
=
(\ddc u)^n
$
in $\Omega$. The comparison principle gives $v_K\le u$
in $\Omega$.

We now apply Lemma~\ref{lem-local-holder-boundary-barrier} to $\phi$ and
$\Gamma_{r_K/2}$. It gives a barrier $h_{-\phi,\Gamma_{r_K/2}}$ such that
$$
u_\phi\le -h_{-\phi,\Gamma_{r_K/2}}
\, \text{ in } \Omega \,\text{ and }
-h_{-\phi,\Gamma_{r_K/2}}(\xi)=\phi(\xi)
\text{ for }
\xi\in\Gamma_{r_K/2},
$$
where $u_\phi:=u(\Omega,\phi,0)$.
Since
$(\ddc u)^n=f\beta^n\ge0=(\ddc u_\phi)^n$
and $u=u_\phi=\phi$ on $\partial\Omega$, the comparison principle gives
$u\le u_\phi$
in $\Omega$. Hence
$v_K\le u\le -h_{-\phi,\Gamma_{r_K/2}}$
in $\Omega$, and both barriers agree with $\phi$ on $\Gamma_{r_K/2}$.
By the Lipschitz estimate for $\widehat b$ and the H\"older estimates from
Lemma~\ref{lem-local-holder-boundary-barrier}, after increasing the constant $C$, one has 
$$
|v_K(z)-v_K(\eta)|
+
|h_{-\phi,\Gamma_{r_K/2}}(z)-h_{-\phi,\Gamma_{r_K/2}}(\eta)|
\le
C \mathcal H_\phi\left(\frac{r_K}{2}\right) d_\beta(z,\eta)^{\alpha/2}
$$
for all $z,\eta\in\overline\Omega$. Hence, for every
$z\in\overline\Omega$ and every $\xi\in\Gamma_{r_K/2}$, we have
$|u(z)-u(\xi)|
\le
C \mathcal H_\phi\left(\frac{r_K}{2}\right)d_\beta(z,\xi)^{\alpha/2}$.
Therefore, for every $\delta>0$, if $x,y\in \overline{\Omega}$ and $\xi\in\Gamma_{r_K/2}$ satisfy
$d_\beta(x,\xi)<\delta$ and 
$d_\beta(y,\xi)<\delta$,
then 
\begin{equation}\label{middle_result}
    |u(x)-u(y)|\le |u(x)-u(\xi)|+|u(y)-u(\xi)| \le 
 2C  \mathcal H_\phi\left(\frac{r_K}{2}\right)\delta^{\frac{\alpha}{2}}.
\end{equation}

Set $\delta_K:=\frac12\lambda_K$.
Let $x,y\in K$ and $0<s<\delta_K$ satisfy
$d_\beta(x,\partial\Omega)<s, d_\beta(y,\partial\Omega)<s, d_\beta(x,y)<s$.
Choose $\xi\in\partial\Omega$ such that
$d_\beta(x,\xi)=d_\beta(x,\partial\Omega)$.
By the first claim, we have
$\xi\in\Gamma_{r_K/2}$
and $d_\beta(x,\xi)<s$.
By the triangle inequality,
$d_\beta(y,\xi)
\le
d_\beta(y,x)+d_\beta(x,\xi)
<2s$.
Thus
$$
|u(x)-u(y)|
\le
|u(x)-u(\xi)|+|u(y)-u(\xi)|
\le
\left(C \mathcal H_\phi\left(\frac{r_K}{2}\right)\right)(s^{\alpha/2}+(2s)^{\alpha/2})
$$
Let $C_0:=(1+2^{\frac{\alpha}{2}})C$, we obtain
$|u(x)-u(y)|
\le
C_0  \mathcal H_\phi\left(\frac{r_K}{2}\right)  s^{\alpha/2}$.
The final assertion follows from the first claim, because
$\delta_K<\lambda_K$.
\end{proof}

\subsection{Peak sets and plurisubharmonic barriers}\label{peak function}---
Peak functions are classical tools in several complex variables. We use here the
usual notion of a peak set for a holomorphic function algebra; see
\cite[p.~385, Section~0]{Noell1986}. For peak functions at boundary points, see
also \cite[p.~555]{BedfordFornaess1978}.
\begin{definition}
Let $E\subset  \partial\Omega$ be a compact set. We say that $E$ admits a
holomorphic peak function if there exists an open neighborhood $U$ of
$\overline\Omega$ and a function $P\in\mathcal O(U)$ such that
$$
P=1 \quad \text{on } E \,
\text{ and } \,
|P|<1 \quad \text{on } \overline\Omega\setminus E.
$$
In this case, $E$ is called a holomorphic peak set for $\Omega$.
\end{definition}

\begin{prop}
\label{prop:peak-function-barrier}
Let $E$ be a compact holomorphic peak set for $\Omega$. 
Then,  $\forall K\Subset\overline\Omega\setminus E$, there exist positive constants
$c_0$, $c_K$ and $\lambda_K$ such that, for every
$0<\lambda<\lambda_K$, there exists a function
$\psi_\lambda\in{\rm PSH}(\Omega)\cap C^0(\overline\Omega)$ satisfying
$$
-1\le \psi_\lambda\le 0
\;
\text{ on } \Omega,\quad
\psi_\lambda\le -c_0
\;
\text{ on } \Omega\cap N_\lambda E,\quad 
\psi_\lambda\ge -c_K\lambda
\;
\text{ on } K\cap\Omega,
$$
where
$N_\lambda E:=\{x\in X:\ d_\beta(x,E)<\lambda\}.$
\end{prop}

\begin{proof}
Let $P$ be a holomorphic peak function for $\Omega$, and set
$H:=1-P.$
Since  $|P|\le 1$ on $\overline\Omega$, we have
${\rm Re}\,H=1-{\rm Re}\,P\ge0$
on $\overline{\Omega}$.
Moreover, since $P$ is holomorphic in a neighborhood of $\overline\Omega$ and
$P=1$ on $E$, there exist constants $A>0$ and $\varepsilon_0>0$ such that
$$
|H(x)|=|1-P(x)|
\le
A d_\beta(x,E)\quad \text{ for }d_\beta(x,E)<\varepsilon_0.
$$
For $0<\lambda\ll1$, define
$$
\psi_\lambda(x)
:=
-{\rm Re}
\left(
\frac{2A\lambda}{2A\lambda+H(x)}
\right)= -
\frac{2A\lambda(2A\lambda+{\rm Re}\left(H(x)\right) }{|2A\lambda+H(x)|^2}.
$$
Since the denominator $2A\lambda+H$ does not
vanish on $\Omega$, one has
$\psi_\lambda\in{\rm PSH}(\Omega)\cap C^0(\overline\Omega).$

Write $H=u+iv$, with $u={\rm Re}\,H\ge0$. Then
$
-1\le \psi_\lambda
=
\frac{-2A\lambda(2A\lambda+u)}
{(2A\lambda+u)^2+v^2} <0$ on $ \Omega$.
If $x\in\Omega\cap N_\lambda E$ and $\lambda<\varepsilon_0$, then
$|H(x)|
\le
A\lambda$.
Hence
$|2A\lambda+H(x)|
\le 3A\lambda$
while
$2A\lambda+{\rm Re}\,H(x)\ge 2A\lambda.$
Therefore
$$
\psi_\lambda(x)
=-
\frac{2A\lambda(2A\lambda+{\rm Re}\,H(x))}
{|2A\lambda+H(x)|^2}
\le
-\frac{4A^2\lambda^2}{9A^2\lambda^2}
=
-\frac49\qquad
\text{on } \Omega\cap N_\lambda(E).
$$
Thus we may take $c_0=4/9$.
It remains to estimate $\psi_\lambda$ on $K$. Since
$K\Subset\overline\Omega\setminus E$ and $|P|<1$ on
$\overline\Omega\setminus E$, there exists $b_K>0$ such that
${\rm Re}\,H\ge b_K$
on $K$.
Therefore, on $K\cap\Omega$,
$$
 \psi_\lambda
\ge -
\left|
\frac{2A\lambda}{2A\lambda+H}
\right|
\ge
\frac{-2A\lambda}{{\rm Re}(2A\lambda+H)}
\ge -
\frac{2A}{b_K}\lambda.
$$
Taking $c_K=\frac{2A}{b_K}$, 
the proof is complete.
\end{proof}

\begin{prop}
\label{prop:finite-singular-boundary-peak}
Assume that $S=X_{\rm sing}\cap\partial\Omega=\{a_1,\ldots,a_m\}$
is finite. Then $S$ is a holomorphic peak set for $\Omega$.
\end{prop}

\begin{proof}
Fix $a_j\in S$. Since $\Omega$ is strongly pseudoconvex, after choosing a
local embedding of a neighborhood $U_j$ of $a_j$ in $X$ into some
$\mathbb C^{N}$ and translating the coordinates, we assume that
$a_j=0$. Then near $a_j$,
$\Omega=\{\rho_j<0\}$ and $\rho_j(0)=0$, 
where $\rho_j$ is the restriction of the defining smooth strictly
plurisubharmonic function $\rho$ on $U_j$. Let
$z=(z_1,\ldots,z_N)$ be the ambient coordinates. Define the holomorphic
polynomial
$$
L_j(z)
:=
\sum_{l=1}^{N}
\frac{\partial\rho_j}{\partial z_l}(0)z_l
+
\frac12
\sum_{l,k=1}^{N}
\frac{\partial^2\rho_j}{\partial z_l\partial z_k}(0)
z_l z_k .
$$
Set $F_j:=-L_j|_{X\cap U_j}.$
Thus $F_j$ is holomorphic on $X\cap U_j$ and $F_j(a_j)=0$.
By Taylor's formula,
$
\rho_j(z)
=
2{\rm Re}\,L_j(z)
+
\sum_{l,k=1}^{N}
\frac{\partial^2\rho_j}{\partial z_l\partial\overline z_k}(0)
z_l\overline z_k
+
o(|z|^2).
$
Since $\rho_j$ is strictly plurisubharmonic, there exists a constant $c_j>0$ such that
$\rho_j(z)
\ge
2{\rm Re}\,L_j(z)+c_j|z|^2$
. Hence, for $x\in\overline\Omega\cap U_j$, since
$\rho_j(x)\le0$, we obtain
$
{\rm Re}\,F_j(x)
=
-{\rm Re}\,L_j(x)
\ge
\frac{c_j}{2}|x|^2 .
$
In particular, $
F_j(a_j)=0$, $
{\rm Re}\,F_j\ge0$ on $\overline\Omega\cap U_j$ and $
{\rm Re}\,F_j>0$ on $(\overline\Omega\cap U_j)\setminus\{a_j\}.
$

After shrinking the neighborhoods $U_j$, we assume that they are pairwise
disjoint. Choose open neighborhoods
$a_j\in V_j\Subset U_j$ for $
j=1,\ldots,m.$
Since $F_j$ does not vanish on the compact set
$\overline\Omega\cap(\overline U_j\setminus V_j)$, we may choose
$\varepsilon>0$ sufficiently small such that
$\Omega^\varepsilon:=\{\rho<\varepsilon\}$
is a Stein neighborhood of $\overline\Omega$ and
$F_j\ne0$ on $\Omega^\varepsilon\cap(U_j\setminus \overline V_j)$
for every $j$.

Set
$U_0:=\Omega^\varepsilon\setminus\bigcup_{j=1}^m \overline V_j$ and $
U_j':=\Omega^\varepsilon\cap U_j$.
Then $\{U_0,U_1',\ldots,U_m'\}$
is an open cover of $\Omega^\varepsilon$. On
$U_0\cap U_j'$ the function $1/F_j$ is holomorphic. They form an additive Cousin datum on the Stein space $\Omega^\varepsilon$.
Since $H^1(\Omega^\varepsilon,\mathcal O)=0$ by Cartan's Theorem~B, this datum is a coboundary. Hence there exist holomorphic functions 
$g_0\in\mathcal O(U_0)$, $
g_j\in\mathcal O(U_j')$, for $
j=1,\ldots,m,
$
such that
$g_j-g_0=\frac1{F_j}$ on $U_0\cap U_j'$.

Define a holomorphic function $G$ on $\Omega^\varepsilon\setminus S$ by
$G=\dfrac1{F_j}-g_j$ on $U_j'\setminus\{a_j\}$ and $G=-g_0$ on $U_0$.
Then $G$ is well defined and holomorphic on
$\Omega^\varepsilon\setminus S$.
We claim that there exists a constant $M>0$ such that
${\rm Re}\,(G+M)>0$ on  $\overline\Omega\setminus S $.

Indeed, near $a_j$ one has
$G=\frac1{F_j}-g_j,$
and, on $\overline\Omega\cap U_j\setminus\{a_j\}$,
${\rm Re}\,\frac1{F_j}
=
\frac{{\rm Re}\,F_j}{|F_j|^2}
>0.$
Since $g_j$ is holomorphic near $a_j$, it is bounded there. Away from the
points $a_j$, the function $G$ is holomorphic in a neighborhood of the relevant
compact subset of $\overline\Omega$. Therefore ${\rm Re}\,G$ is bounded from
below on $\overline\Omega\setminus S$, and the claim follows after choosing
$M$ sufficiently large.

Set
$H:=\frac1{G+M}$
on $\Omega^\varepsilon\setminus S$. Since ${\rm Re}\,(G+M)>0$ on
$\overline\Omega\setminus S$, the denominator does not vanish there. Moreover,
near $a_j$ we have
$H=\frac{1}{\frac1{F_j}-g_j+M}=
\frac{F_j}{1+F_j(M-g_j)}.$
Thus $H$ extends holomorphically across $a_j$ by setting $H(a_j):=0$. After
shrinking the neighborhood of $\overline\Omega$ if necessary, $H$ is
holomorphic in a neighborhood of $\overline\Omega$.

Furthermore, on $\overline\Omega\setminus S$,
${\rm Re}\,H
=
{\rm Re}\left(\frac1{G+M}\right)
>0$.
Define $P:=\exp(-H).$
Then $P$ is holomorphic in a neighborhood of $\overline\Omega$. Since $H=0$ on $S$, we have
$P=1$ on $S$.
On the other hand, for $x\in\overline\Omega\setminus S$,
$|P(x)|
=
\exp(-{\rm Re}\,H(x))
<1.$
Thus $S$ is a holomorphic peak set for $\Omega$.
\end{proof}

We now give  non-finite singular boundary sets which are 
holomorphic peak sets.

\begin{example}
\label{ex:nonfinite-peak-singular-boundary}
Fix $k\ge2$,
let
$X_k:=
\{(z,v,w)\in\mathbb C^k\times\mathbb C\times\mathbb C:w^2=v(1-Q_k(z))^3\}$,  where  $
z=(z_1,\ldots,z_k)\in\mathbb C^k$ and 
$Q_k(z):=k^{k/2}z_1\cdots z_k$.
Then $X_k$ is a reduced locally irreducible Stein space in $\mathbb C^{k+2}$. 
Set
$\Omega_k
:=
X_k\cap
\left\{
\sum_{j=1}^k
|z_j|^2+|v|^2+|w|^2<1
\right\}.$
Since
$\rho(z,v,w):= \left(\sum_{j=1}^k
|z_j|^2+|v|^2+|w|^2-1\right)$
is strictly plurisubharmonic in the ambient space, $\Omega_k$ is a strongly
pseudoconvex domain in $X_k$.
One has
$(X_k)_{\rm sing}
=
\{w=0,\ Q_k(z)=1\}$ and 
$S_k
:=
\partial\Omega_k\cap (X_k)_{\rm sing}
=
\left\{
\left(
\frac{e^{i\theta_1}}{\sqrt k},
\ldots,
\frac{e^{i\theta_k}}{\sqrt k},0,
0
\right):
\theta_1+\cdots+\theta_k\equiv0 \mod 2\pi
\right\}$.

Now define
$P_k(z,v,w):=\frac{1+Q_k(z)}{2}.$
Then $P_k$ is holomorphic in a neighborhood of $\overline{\Omega_k}$. On
$\overline{\Omega_k}$, the arithmetic-geometric mean inequality gives
$
|Q_k(z)|
\le
\left(|z_1|^2+\cdots+|z_k|^2\right)^{k/2}
\le
1
$
with equality precisely when
$|z_1|=\cdots=|z_k|=\frac1{\sqrt k}$. Therefore
$|P_k|
=
\left|\frac{1+Q_k}{2}\right|
\le1.$
Moreover, equality $|P_k|=1$ can occur only when $Q_k=1$. On $X_k$, this
implies $w=0$. The condition $(z,v,w)\in\overline{\Omega_k}$ and $|Q_k(z)|=1$ force $v=0$ and 
$
|z_1|=\cdots=|z_k|=\frac1{\sqrt k}.
$
Thus, on $\overline\Omega_k$, one has $P_k=1$ precisely on $S_k$,
and
$|P_k|<1$ on $\overline{\Omega_k}\setminus S_k$.

Hence $S_k$ is a positive real dimensional holomorphic peak set for $\Omega_k$.
\end{example}

\section{H\"older continuity}
In this section, we estimate the local H\"older modulus of continuity on compact subsets of
$\overline\Omega\cap X_{\rm reg}$ for
the solution $u$ of the Dirichlet problem \eqref{main-eq-general} .

\subsection{Stability estimate}
---
We consider resolution
$\pi:\widetilde X\to X$
which is biholomorphic over $X_{\mathrm{reg}}$, such that
$\widetilde X$ is K\"ahler and
$E:=\pi^{-1}(X_{\mathrm{sing}})$
is a divisor.   Set
$\widetilde\Omega:=\pi^{-1}(\Omega)$ and
$\widetilde u:=u\circ\pi .$
Fix a K\"ahler form $\tau$ on $\widetilde X$. After replacing $\tau$ by a sufficiently
large positive multiple, we may assume that
$\pi^*\beta\le \tau$ on 
$\pi^{-1}(\overline\Omega)$.
We will denote objects on $\widetilde X$ with a tilde.
We now recall Demailly's regularization in \cite[Section 8]{Demailly1982}. 
Set
$$\exp : T\widetilde X \longrightarrow \widetilde X,\qquad
T\widetilde X \ni (\widetilde x,\xi)\longmapsto
\exp_{\widetilde x}(\xi)\in \widetilde X,\quad
\xi\in T_{\widetilde x}\widetilde X . \ $$
be the exponential map associated with $\tau$. The exponential map $\exp$ is a $C^\infty $ mapping. For $0<\delta<\delta_0$, set
$\widetilde\Omega_\delta
:=
\left\{
\widetilde x\in\widetilde\Omega ;
d_\tau(\widetilde x,\partial\widetilde\Omega)>\delta
\right\},$
where $\delta_0>0 $ is fixed so that $\widetilde\Omega_\delta \neq \varnothing$. 
Similarly, we set $\Omega_\delta
:=
\left\{
 x\in\Omega ;
d_\beta(x,\partial\Omega)>\delta
\right\}$ for $\delta$ small enough. For $\delta>0$ sufficiently small and $\widetilde x\in\widetilde\Omega_\delta$, we define the $\delta$-regularization of $\widetilde u$ by
$$
\eta_\delta\widetilde u(\widetilde x)
:=
\frac{1}{\delta^{2n}}
\int_{\xi\in T_{\widetilde x}\widetilde X}
\widetilde u\bigl(\exp_{\widetilde x}(\xi)\bigr)
\eta\left(\frac{|\xi|_\tau^2}{\delta^2}\right)
\,dV_\tau(\xi),
$$
where $\eta$ is a smoothing kernel, 
$|\xi|_\tau^2
=
\sum_{i,j=1}^n \tau_{i\bar j}(\widetilde x)\xi_i\overline{\xi_j}$ and 
$dV_\tau(\xi)
=
\frac{1}{2^n n!}\bigl(dd^c|\xi|_\tau^2\bigr)^n .$

We fix once and for all a smooth radial kernel
$\eta\in C^\infty([0,+\infty))$ such that
$\eta\ge 0$, $\rm{supp}\eta\subset [0,1)$ and $
\int_{\mathbb C^n}\eta(|\zeta|^2)\,dV(\zeta)=1$.
 The dependence on the scale
$\delta$ enters only through the rescaling
$\eta(|\xi|_\tau^2/\delta^2)$.
The normalization is independent of $\widetilde x$, since after choosing a
$\tau$-unitary frame on $T_{\widetilde x}\widetilde X$ and making the change of
variables $\xi=\delta\zeta$, one obtains
$$
\frac{1}{\delta^{2n}}
\int_{T_{\widetilde x}\widetilde X}
\eta\left(\frac{|\xi|_\tau^2}{\delta^2}\right)dV_\tau(\xi)
=
\int_{\mathbb C^n}\eta(|\zeta|^2)dV(\zeta)
=
1.
$$
Following Demailly's notation, we set
$
U(\widetilde x,w):=\eta_{|w|}\widetilde u(\widetilde x)$, $ w\in\mathbb C.$
Then we recall the regularization estimate in \cite[Lemma~2.1]{DDGPKZ14}.
\begin{lemma}
\label{lem-demailly-kiselman}
Let $\widetilde{u}$ be  a bounded psh function on the  K\"ahler manifold $(\widetilde{ \Omega},\tau)$.
There exists a constant $J>0$, depending on the curvature of $(\overline{\widetilde{\Omega}},\tau)$, such that for every fixed $\widetilde x$, the function
$
t\mapsto U(\widetilde x,t)+Jt^2
=
\eta_t\widetilde u(\widetilde x)+Jt^2
$
is increasing for $t$. For $c>0$ and $\delta>0$ sufficiently small, define the Kiselman--Legendre transform
of $\widetilde u$ by
\begin{equation}\label{u_c_delta}
\widetilde u_{c,\delta}(\widetilde x)
:=
\inf_{0<t\le \delta}
\left[
\eta_t\widetilde u(\widetilde x)
+
Jt^2
-
J\delta^2
-
c\log\frac{t}{\delta}
\right],
\qquad
\widetilde x\in\widetilde\Omega_\delta .
\end{equation}
Then we have the following estimate for the complex Hessian
$$
dd^c\widetilde u_{c,\delta}
\ge
-\bigl(A_1c+J\delta\bigr)\tau
$$
in the sense of currents on $\widetilde\Omega_\delta$, where  $A_1\ge 1$ is a lower bound for the negative part of the holomorphic bisectional
curvature of $(\overline{\widetilde{ \Omega}},\tau)$.
\end{lemma}

For $0<\alpha_*<1$, we have
$\delta^{\alpha_*}\gg \delta$ as  $\delta\to0^+.$ Thus, after decreasing $\delta$ if necessary, we may assume that
\begin{equation}\label{c_delta}
c_\delta
:=
\frac{\delta^{\alpha_*}-J\delta}{A_1}>\frac{\delta^{\alpha_*}}{2A_1}>0.
\end{equation}
Then  we obtain
$
A_1c_\delta+J\delta=\delta^{\alpha_*}
$
and
$dd^c\widetilde u_{c_\delta,\delta}
\ge
-\delta^{\alpha_*}\tau$ on $ \widetilde\Omega_\delta$ .

Since $\Omega$ is strongly pseudoconvex, after shrinking the neighborhood on
which $\rho$ is defined, we may choose $\varepsilon>0$ small enough so that
$
\widehat\Omega:=\{\rho<\varepsilon\}
$
satisfies
$
\Omega\Subset \widehat\Omega\Subset X,
$
and $\rho$ is smooth and strongly plurisubharmonic near $\overline{\widehat\Omega}$. Set
$\widetilde{\widehat\Omega}:=\pi^{-1}(\widehat\Omega).$

By the desingularization theorem originating in
Hironaka's resolution,
 there exist positive rational
numbers
$b_i\in\mathbb Q_{+}$
such that the exceptional
$\mathbb Q$-divisor
$-\sum_i b_iE_i$
satisfies
$\mathrm{supp}(-\sum_i b_iE_i)\cap\widetilde{\widehat\Omega}
=
E\cap\widetilde{\widehat\Omega}
$
and
$-\sum_i b_iE_i$
is $\pi$-ample, where each exceptional divisor $E_i$ is an irreducible component of $E$.
For each $i$, let $s_i$ be the canonical section of
$\mathcal O_{\widetilde X}(E_i)$, and choose a smooth Hermitian metric $h_i$
on $\mathcal O_{\widetilde X}(E_i)$ over
$\widetilde{\widehat\Omega}$. Since $-\sum_i b_iE_i$ is $\pi$-ample over $\widehat\Omega$, the metrics
$h_i$ can be chosen so that the smooth real $(1,1)$-form
$
\vartheta_E
:=
-\sum_i b_i\,
\Theta_{h_i}\bigl(\mathcal O_{\widetilde X}(E_i)\bigr)
$
is positive on
$\ker d\pi$ over $\widetilde{\widehat\Omega}$.

Set
$\widetilde {\overline\Omega}:=\pi^{-1}(\overline\Omega)\subset \widetilde{\widehat\Omega}$.
Since $\pi$ is proper and $\overline\Omega$ is compact, $\widetilde {\overline\Omega}$ is compact. 
After replacing $-\sum_i b_iE_i$ by a sufficiently large positive multiple, and keeping the
same notation, we may assume that
$\vartheta_E\ge 2\tau$ on $\ker d\pi\text{ over }\widetilde {\overline\Omega}$. 
By compactness of $\widetilde {\overline\Omega}$,
there exists $A>0$ such that
$\vartheta_E+A\pi^*dd^c\rho\ge \tau $ on $\widetilde {\overline\Omega}$. Define
$
\widetilde\rho_E^0
:= A\pi^*\rho+ \sum_i b_i
\log |s_i|_{h_i}^2 .
$
Then $\widetilde\rho_E^0$ is smooth on $\widetilde\Omega\setminus E$, has
logarithmic poles along $E$, and satisfies
$\widetilde\rho_E^0\to-\infty$ along $E$.
Moreover, we have
$dd^c\widetilde\rho_E^0\ge \tau $ on $\widetilde\Omega$
in the sense of currents.

Finally, $\widetilde\rho_E^0$ is bounded from above on $\widetilde\Omega$. Hence, after
subtracting a sufficiently large constant $C$, we obtain
$\widetilde\rho_E:=\widetilde\rho_E^0-C\le 0$ on $\widetilde\Omega$.
Subtracting a constant does not change its current, so
$
dd^c\widetilde\rho_E\ge \tau$ on $\widetilde\Omega.$
Combining \eqref{u_c_delta} and \eqref{c_delta}, we define
$$
\widehat u_\delta
:=
\widetilde u_{c_\delta,\delta}
+
\delta^{\alpha_*}\widetilde\rho_E
\quad\text{on } \widetilde\Omega_\delta .
$$
Then
$
dd^c\widehat u_\delta
=
dd^c\widetilde u_{c_\delta,\delta}
+
\delta^{\alpha_*}dd^c\widetilde\rho_E
\ge
-\delta^{\alpha_*}\tau
+
\delta^{\alpha_*}\tau
=
0.
$
Hence $\widehat u_\delta$ is a singular plurisubharmonic function on $\widetilde\Omega_\delta$.
Moreover, since $\widetilde\rho_E\to-\infty$ along $E$ and $\widetilde u_{c_\delta,\delta}$ is bounded from above locally, the function $\widehat u_\delta$ tends to
$-\infty$ along $E$.
On $\Omega_\delta\cap X_{\mathrm{reg}}$, the map $\pi$ is biholomorphic. For $x\in \Omega_\delta\cap X_{\mathrm{reg}}$, we therefore
define
\begin{equation}\label{check_u_delta}
\check u_\delta(x)
:=
\widehat u_\delta(\pi^{-1}(x)).
\end{equation}
Since $\widehat u_\delta$ tends to $-\infty$ along $E=\pi^{-1}(X_{\mathrm{sing}})$,
we have
$\limsup_{\Omega_\delta\cap X_{\mathrm{reg}}\ni x\to a}
\check u_\delta(x)
=
-\infty$
for every $a\in\Omega_\delta\cap X_{\mathrm{sing}}$.
Since $\pi^*\beta \le \tau$ on $\widetilde {\overline{\Omega}}$, if $z\in \Omega_\delta$, then  we get $\widetilde{z} \in \widetilde \Omega_\delta$. 
Thus the upper semicontinuous extension of $\check u_\delta$ across
$X_{\mathrm{sing}}$, obtained by setting it equal to $-\infty$ on
$\Omega_\delta\cap X_{\mathrm{sing}}$, is plurisubharmonic on $\Omega_\delta$
by Theorem~\ref{weak_psh_psh}. We still denote this extension by
$\check u_\delta$.

\begin{prop}
\label{prop:general-modulus-capacity-transfer}
Let $u$ be the solution given by Theorem~\ref{main thm}. Assume that
$\phi$ is locally $\alpha$-H\"older continuous on
$\partial\Omega\cap X_{\rm reg}$ and
$S$ is a holomorphic peak set for $\Omega$. Fix $r>0$, $0<\kappa<1$ and $0<\alpha_*<\frac{1}{nq+1}.$
Set $
\Omega_r^{\rm reg}
:=
\{x\in\overline{\Omega}:\ d_\beta(x,X_{\rm sing})\ge r\}.$
Then there exist constants $C_r>0$ and $\delta_{\kappa
}>0$ such that, for all
$0<\delta<\delta_{\kappa
}$,
$$
\check u_\delta-u
\le
C_r
\left[
\delta^{\alpha_*}
+
\mathcal H_{\phi}\left(\frac{\delta^\kappa}{2}\right)
\delta^{\frac{\alpha}{2}}
+
\delta^\kappa
\right]
\quad
\text{on }
\Omega_{2\delta}\cap\Omega_r^{\rm reg}.
$$
\end{prop}

\begin{proof}
Choose
$\gamma\in\left(\alpha_*,\frac{1}{nq+1}\right).$
Then choose $0<\varepsilon<1$ such that
$\gamma(1-\varepsilon)>\alpha_*.$
Set
$
\lambda_\delta:=\delta^\kappa.
$
Since $0<\kappa<1$, after decreasing $\delta_{\kappa
}>0$ if necessary, we may assume
that
$8\delta\le \lambda_\delta$
for all $0<\delta<\delta_{\kappa
}<1$.

By Proposition~\ref{prop:peak-function-barrier}, applied to the holomorphic
peak set $S$ and to the compact set containing $\Omega_r^{\rm reg}$, there exist
constants $c_0>0$, $c_r>0$, and $\lambda_r>0$ such that, for every
$0<\lambda<\lambda_r$, there exists a function
$\psi_\lambda\in{\rm PSH}(\Omega)\cap C^0(\overline\Omega)$ satisfying
$$
-1\le \psi_\lambda\le 0
\;
\text{ on } \Omega,\quad
\psi_\lambda\le -c_0
\;
\text{ on } \Omega\cap N_\lambda S,\quad 
\psi_\lambda\ge -c_r\lambda
\;
\text{ on } \Omega_r^{\rm reg},
$$
After decreasing $\delta_{\kappa}$ if necessary, we may assume that $8\delta \le \lambda_{\delta}<\lambda_r$ for all 
$0<\delta<\delta_{\kappa
}<1$. Set $\psi_\delta:=\psi_{\lambda_\delta}.$
Let
$L:=\frac{2\|u\|_{L^\infty(\Omega)}}{c_0}$
and define
$V_{\delta,r}:=\check u_\delta+L\psi_\delta$ on $\Omega_{2\delta}$.
Since both $\check u_\delta$ and $\psi_\delta$ are plurisubharmonic on
$\Omega_{2\delta}$, we have
$
V_{\delta,r}\in{\rm PSH}(\Omega_{2\delta}).
$
We claim that there exists a constant $C_r>0$ such that, for all
$0<\delta<\delta_{\kappa
}$,
$$
\limsup_{\Omega_{2\delta}\ni z\to\zeta}
V_{\delta,r}(z)
\le
u(\zeta)
+
C_r
\mathcal H_{\phi}
\left(
\frac{\lambda_\delta}{2}
\right)
\delta^{\alpha/2} \quad \text{ for every } \zeta\in\partial\Omega_{2\delta}. 
$$
It is enough to prove this estimate for points
$z\in\Omega_{2\delta}$ satisfying
$
d_\beta(z,\partial\Omega)\le 3\delta,
$
and then let $z\to\zeta\in\partial\Omega_{2\delta}$.

We distinguish two cases. First assume that
$z\in N_{\lambda_\delta}S.$
Since $\psi_\delta\le -c_0$ on $\Omega\cap N_{\lambda_\delta}S$, we get
$
V_{\delta,r}(z)
=
\check u_\delta(z)+L\psi_\delta(z)
\le
\check u_\delta(z)-Lc_0.
$
By the definition of $\check u_\delta$ and since $\rho_E\le0$, we have
$\check u_\delta\le \pi_*(\eta_\delta\widetilde u)
\le
\|u\|_{L^\infty(\Omega)}.$
Therefore,
$
V_{\delta,r}(z)
\le
\|u\|_{L^\infty(\Omega)}-Lc_0
=
-\|u\|_{L^\infty(\Omega)}
\le
u(z).
$

Now assume that
$z\notin N_{\lambda_\delta}S.$
Choose $\xi\in\partial\Omega$ such that
$d_\beta(z,\xi)=d_\beta(z,\partial\Omega).$
Since $d_\beta(z,S)\ge\lambda_\delta$ and $d_\beta(z,\xi)\le3\delta$, the
condition $8\delta\le\lambda_\delta$ gives
$
d_\beta(\xi,S)
\ge
d_\beta(z,S)-d_\beta(z,\xi)
\ge
\lambda_\delta-3\delta
\ge
\frac{\lambda_\delta}{2}.
$
Thus
$\xi\in \Gamma_{\lambda_\delta/2}\subset \partial\Omega\cap X_{\rm reg}.$
Let $\widetilde z=\pi^{-1}(z)$ whenever $z\in X_{\rm reg}$. If $z$ is a
singular point, the following argument is applied by taking limits from
$X_{\rm reg}$. For every point $\widetilde y\in\widetilde\Omega$ satisfying
$d_\tau(\widetilde y,\widetilde z)\le\delta,$
since $\pi^*\beta\le\tau$, it follows that
$d_\beta(\pi(\widetilde y),z)\le\delta.$
Consequently,
$
d_\beta(\pi(\widetilde y),\xi)
\le
d_\beta(\pi(\widetilde y),z)+d_\beta(z,\xi)
\le
4\delta.
$
By the intermediate estimate \eqref{middle_result} in the proof of Proposition~\ref{near boundary K}, applied with
$\overline{\Omega}\setminus N_{\lambda_\delta}S$ and $4\delta$, we obtain
$$
\bigl|
u(\pi(\widetilde y))-u(z)
\bigr|
\le
2C
\mathcal H_{\phi}
\left(
\frac{\lambda_\delta}{2}
\right)
(4\delta)^{\alpha/2}= 2^{\alpha+1}C\mathcal H_{\phi}
\left(
\frac{\lambda_\delta}{2}
\right)
\delta^{\alpha/2}.
$$
Set $C_1:=2^{\alpha+1}C$. 
Averaging in the regularization variable gives
$
\eta_\delta\widetilde u(\widetilde z)
\le
u(z)
+
C_1
\mathcal H_{\phi}
\left(
\frac{\lambda_\delta}{2}
\right)
\delta^{\alpha/2}.
$
Since
$\check u_\delta\le \pi_*(\eta_\delta\widetilde u)$
and $\psi_\delta\le0$, we obtain
$$
V_{\delta,r}(z)
=
\check u_\delta(z)+L\psi_\delta(z)
\le
\check u_\delta(z)
\le
u(z)
+
C_1
\mathcal H_{\phi}
\left(
\frac{\lambda_\delta}{2}
\right)
\delta^{\alpha/2}.
$$
Define
$\Psi_{\delta,r}
:=
u+
C_1
\mathcal H_{\phi}
\left(
\frac{\lambda_\delta}{2}
\right)
\delta^{\alpha/2},$
and set
$$
U_{\delta,r}
:=
\begin{cases}
\max\{V_{\delta,r},\Psi_{\delta,r}\},
& \text{on }\Omega_{2\delta},\\
\Psi_{\delta,r},
& \text{on }\Omega\setminus\Omega_{2\delta}.
\end{cases}
$$
By the boundary estimate just proved, we have
$U_{\delta,r}\in{\rm PSH}(\Omega)\cap L^\infty(\Omega).$

On $\Omega\setminus\Omega_{2\delta}$, the definition gives
$\bigl(U_{\delta,r}-\Psi_{\delta,r}\bigr)_+=0.$
On $\Omega_{2\delta}$, we have
$U_{\delta,r}
=
\max\{V_{\delta,r},\Psi_{\delta,r}\}.$
Therefore, $\bigl(U_{\delta,r}-\Psi_{\delta,r}\bigr)_+
=
\bigl(V_{\delta,r}-\Psi_{\delta,r}\bigr)_+.$
Since
$V_{\delta,r}
=
\check u_\delta+L\psi_\delta
\le
\check u_\delta$
and
$\Psi_{\delta,r}\ge u,$
we get
$
\bigl(U_{\delta,r}-\Psi_{\delta,r}\bigr)_+
\le
(\check u_\delta-u)_+$ on $\Omega_{2\delta}.$
Hence
$$
\left\|
\bigl(U_{\delta,r}-\Psi_{\delta,r}\bigr)_+
\right\|_{L^1(\Omega,\beta^n)}
\le
\int_{\Omega_{2\delta}}
(\check u_\delta-u)_+\,\beta^n.
$$
Since
$\check u_\delta
=
\pi_*
\left(
\widetilde u_{c_\delta,\delta}
+
\delta^{\alpha_*}\widetilde\rho_E
\right),$
and since $\widetilde\rho_E\le0$ and
$\widetilde u_{c_\delta,\delta}\le \eta_\delta\widetilde u,$
we obtain, outside $X_{\rm sing}$,
$(\check u_\delta-u)_+
\le
\pi_*
\bigl(\eta_\delta\widetilde u-\widetilde u\bigr)_+.$
Since $X_{\rm sing}$ has zero $\beta^n$-measure, using $\pi^*\beta\le\tau$, we
obtain
$$
\int_{\Omega_{2\delta}}(\check u_\delta-u)_+\,\beta^n
\le
\int_{\pi^{-1}(\Omega_{2\delta})}
\bigl(\eta_\delta\widetilde u-\widetilde u\bigr)_+
\,\tau^n.
$$
Applying the local $L^1$-estimate \cite[Lemma~4.1]{CerqueiraGoncalves2026} in finitely
many coordinate charts covering $\pi^{-1}(\overline\Omega)$, we get that there exists $C_\varepsilon>0$ such that
$
\int_{\pi^{-1}(\Omega_{2\delta})}
\bigl(\eta_\delta\widetilde u-\widetilde u\bigr)_+
\,\tau^n
\le
C_\varepsilon\delta^{1-\varepsilon}.
$
Therefore,
$\left\|
\bigl(U_{\delta,r}-\Psi_{\delta,r}\bigr)_+
\right\|_{L^1(\Omega,\beta^n)}
\le
C_\varepsilon\delta^{1-\varepsilon}.$
Applying the stability estimate for fixed densities, \cite[Proposition~1.8]{GuedjGuenanciaZeriahi2023}, to $U_{\delta,r}$ and
$\Psi_{\delta,r}$, we obtain
$$
\sup_\Omega
\bigl(U_{\delta,r}-\Psi_{\delta,r}\bigr)_+
\le
C
\left\|
\bigl(U_{\delta,r}-\Psi_{\delta,r}\bigr)_+
\right\|_{L^1(\Omega,\beta^n)}^\gamma.
$$
Using the $L^1$-estimate above, we get
$\sup_\Omega
\bigl(U_{\delta,r}-\Psi_{\delta,r}\bigr)_+
\le
C_\varepsilon\delta^{\gamma(1-\varepsilon)}.$
Since $\gamma(1-\varepsilon)>\alpha_*$ and $0<\delta<\delta_{\kappa}<1$, it follows that
$\sup_\Omega
\bigl(U_{\delta,r}-\Psi_{\delta,r}\bigr)_+
\le
C_{\varepsilon}\delta^{\alpha_*}$.

It remains to estimate $\check u_\delta-u$ on
$\Omega_{2\delta}\cap\Omega_r^{\rm reg}$. Since
$
V_{\delta,r}\le U_{\delta,r}
$ on $\Omega_{2\delta},$
we have
$\check u_\delta+L\psi_\delta
\le
U_{\delta,r}.$
Thus
$\check u_\delta-u
\le
U_{\delta,r}-u-L\psi_\delta.$
Using the definition of $\Psi_{\delta,r}$, we obtain
$$
\check u_\delta-u
\le
U_{\delta,r}-\Psi_{\delta,r}
+
C_1
\mathcal H_{\phi}
\left(
\frac{\lambda_\delta}{2}
\right)
\delta^{\alpha/2}
+
L(-\psi_\delta).
$$
On $\Omega_r^{\rm reg}$, we have
$-\psi_\delta\le c_r\lambda_\delta=c_r\delta^\kappa.$
Set $C_r:= C_{\varepsilon}+C_1+Lc_r$. 
Then we get
\begin{equation*}
\check u_\delta-u
\le
C_r
\left[
\delta^{\alpha_*}
+
\mathcal H_{\phi}
\left(
\frac{\delta^\kappa}{2}
\right)
\delta^{\alpha/2}
+
\delta^\kappa
\right]
\quad
\text{on }
\Omega_{2\delta}\cap\Omega_r^{\rm reg}. \qedhere
\end{equation*}
\end{proof}

\subsection{Proof of  H\"older continuity}\label{local_holder}

\begin{theorem}
\label{thm:holder-peak-set}
Let $u$ be the solution obtained in Theorem~\ref{main thm}. Assume that
$\phi$ is locally $\alpha$-H\"older continuous on
$\partial\Omega\cap X_{\rm reg}$  and 
$S$ is a holomorphic peak set for $\Omega$.
Assume moreover that there exist constants $B>0$, $\sigma\ge0$, and
$\lambda_0>0$ such that
\begin{equation}\label{mathcal_H_phi_3.4}
\mathcal H_\phi(\lambda)\le B\lambda^{-\sigma},
\qquad
0<\lambda<\lambda_0.
\end{equation}
Then, for every compact set
$K\Subset\overline\Omega\cap X_{\rm reg}$ and every exponent
$0<\nu<
\min\left\{
\frac{1}{nq+1},
\frac{\alpha}{2(\sigma+1)}
\right\},$
there exists a constant $C_{K,\nu}>0$ such that
$$
|u(x)-u(y)|
\le
C_{K,\nu}d_\beta(x,y)^\nu,
\qquad
x,y\in K.
$$
The constant $C_{K,\nu}$ may blow up when $K$ approaches $X_{\rm sing}$.
\end{theorem}

\begin{proof}
Fix
$0<\nu<
\min\left\{
\frac{1}{nq+1},
\frac{\alpha}{2(\sigma+1)}
\right\}.$
Choose $0<\kappa<1$ such that
$\nu<\kappa$ and $\frac{\alpha}{2}-\kappa\sigma>\nu.$
This is possible by the choice of $\nu$. Indeed, if $\sigma=0$, it is enough
to choose any $\kappa\in(\nu,1)$. If $\sigma>0$, the condition
$\nu<\frac{\alpha}{2(\sigma+1)}$
is equivalent to
$\nu<\frac{\alpha/2-\nu}{\sigma},$
so one can choose
$\nu<\kappa<
\min\left\{
1,
\frac{\alpha/2-\nu}{\sigma}
\right\}.$

 Let $r_K=\frac12 d_\beta(K,X_{\rm sing})$. Then $K\cap \overline{\Omega}\subset\Omega_{r_K}^{\rm reg}.$
We apply Proposition~\ref{prop:general-modulus-capacity-transfer} with
$r=r_K$, with exponent $\nu$, and with the above choice of $\kappa$. Then there
exist constants $C_K>0$ and $\delta_{\kappa}>0$ such that, for all
$0<\delta<\delta_{\kappa}$,
$$
\check u_\delta-u
\le
C_K
\left[
\delta^\nu
+
\mathcal H_\phi\left(\frac{\delta^\kappa}{2}\right)
\delta^{\alpha/2}
+
\delta^\kappa
\right]
\quad
\text{on }
K\cap\Omega_{2\delta}.
$$
Using assumption \eqref{mathcal_H_phi_3.4}
for $\delta>0$ sufficiently small, we obtain
$
\mathcal H_\phi\left(\frac{\delta^\kappa}{2}\right)
\delta^{\alpha/2}
\le
2^{\sigma}B\delta^{\alpha/2-\kappa\sigma}.
$
Since
$\kappa>\nu$ and $\frac{\alpha}{2}-\kappa\sigma>\nu,$
after decreasing $\delta_{\kappa}$ and increasing $C_K$, we get
\begin{equation}\label{eq:check-u-holder-bound}
\check u_\delta-u
\le
C_K\delta^\nu
\quad
\text{on }
K\cap\Omega_{2\delta}.
\end{equation}

We now pass from the estimate for $\check u_\delta$ to a H\"older estimate for
$u$ on $K$. Let
$\widetilde K:=\pi^{-1}(K).$
Since $\pi$ is biholomorphic over $X_{\rm reg}$,  $\widetilde\rho_E$ is smooth and bounded on
$\widetilde K$. Thus there exists a constant $C_E(K)>0$ such that
$\widetilde\rho_E\ge -C_E(K)$ on  $\widetilde K$.
Moreover, on $\widetilde K$ the distances $d_\tau$ and $d_\beta$ are uniformly
comparable. Thus there exists $L_K>1$ such that
\begin{equation}\label{eq:distance-comparison-K}
d_\tau(\widetilde x,\widetilde y)
\le
L_Kd_\beta(\pi(\widetilde x),\pi(\widetilde y)),
\qquad
\widetilde x,\widetilde y\in\widetilde K.
\end{equation}

Let $x\in K\cap\Omega_{2\delta}$ and write $\widetilde x:=\pi^{-1}(x)$.
By the definition of $\check u_\delta$, associated with the exponent $\nu$,
we have
$\check u_\delta(x)
=
\widetilde u_{c_\delta,\delta}(\widetilde x)
+
\delta^\nu\widetilde\rho_E(\widetilde x).$
Since $u(x)=\widetilde u(\widetilde x)$, estimate
\eqref{eq:check-u-holder-bound} gives
$
\widetilde u_{c_\delta,\delta}(\widetilde x)
-
\widetilde u(\widetilde x)
\le
C_K\delta^\nu
-
\delta^\nu\rho_E(\widetilde x).
$
Setting $D_K:= C_K+ C_E(K)$, we obtain
\begin{equation}\label{eq:kiselman-bound-K}
\widetilde u_{c_\delta,\delta}
-
\widetilde u
\le
D_K\delta^\nu
\quad
\text{on }
\widetilde K\cap\pi^{-1}(\Omega_{2\delta}).
\end{equation}
Recall that
$
\widetilde u_{c_\delta,\delta}(\widetilde x)
=
\inf_{0<t\le\delta}
\left[
\eta_t\widetilde u(\widetilde x)
+
Jt^2
-
J\delta^2
-
c_\delta\log\frac{t}{\delta}
\right],
$
where
$
c_\delta
=
\frac{\delta^\nu-J\delta}{A_1}\ge \frac{\delta^\nu}{2A_1}
$ from \eqref{c_delta}.
Fix
$\widetilde x\in \widetilde K\cap\pi^{-1}(\Omega_{2\delta}).$
Let $t_\delta\in(0,\delta]$ be a minimizing parameter. Then
$$
\widetilde u_{c_\delta,\delta}(\widetilde x)
=
\eta_{t_\delta}\widetilde u(\widetilde x)
+
Jt_\delta^2
-
J\delta^2
-
c_\delta\log\frac{t_\delta}{\delta}.
$$
Since the function
$t\mapsto \eta_t\widetilde u(\widetilde x)+Jt^2$
is increasing and converges to $\widetilde u(\widetilde x)$ as $t\to0^+$, we have
$
\eta_{t_\delta}\widetilde u(\widetilde x)
+
Jt_\delta^2
-
\widetilde u(\widetilde x)
\ge0.
$
Combining this with \eqref{eq:kiselman-bound-K}, we get
$
-c_\delta\log\frac{t_\delta}{\delta}
\le
D_K\delta^\nu+J\delta^2.
$
Since $c_\delta\ge (2A_1)^{-1}\delta^\nu$, there exists a constant
$A_K'>0$ such that
$
-\log\frac{t_\delta}{\delta}\le A_K'.
$
Thus, for
$
a_K:=e^{-A_K'}\in(0,1),
$
one has
$
t_\delta\ge a_K\delta.
$
By the monotonicity of
$t\mapsto\eta_t\widetilde u(\widetilde x)+Jt^2$, we have
$$
\eta_{a_K\delta}\widetilde u(\widetilde x)
+
Ja_K^2\delta^2
\le
\eta_{t_\delta}\widetilde u(\widetilde x)
+
Jt_\delta^2.
$$
Using again the defining equality for
$\widetilde u_{c_\delta,\delta}$, and using
$\log(t_\delta/\delta)\le0$, we obtain
$
\eta_{a_K\delta}\widetilde u(\widetilde x)
-
\widetilde u(\widetilde x)
\le
\widetilde u_{c_\delta,\delta}(\widetilde x)
-
\widetilde u(\widetilde x)
+
J\delta^2.
$
Together with \eqref{eq:kiselman-bound-K}, this gives
$$
\eta_{a_K\delta}\widetilde u
-
\widetilde u
\le
(D_K+J)\delta^\nu
\quad
\text{on }
\widetilde K\cap\pi^{-1}(\Omega_{2\delta}).
$$
Putting $s=a_K\delta$ for
$0<\delta<\delta_{\kappa}$, we obtain,
\begin{equation}\label{eq:average-holder-K}
\eta_s\widetilde u-\widetilde u
\le
(D_K+J) \frac{1}{a_K^\nu}s^\nu
\quad
\text{on }
\widetilde K\cap\pi^{-1}(\Omega_{2\delta}),
\end{equation}
By Zeriahi's result
\cite[Theorem~3.4]{Zeriahi2020}, estimate
\eqref{eq:average-holder-K} implies that there exists a constant $B_K>0$ such
that, for all
$\widetilde x,\widetilde y
\in
\widetilde K\cap\pi^{-1}(\Omega_{2\delta}),$
one has
$|\widetilde u(\widetilde x)-\widetilde u(\widetilde y)|
\le
B_Kd_\tau(\widetilde x,\widetilde y)^\nu.$
Using \eqref{eq:distance-comparison-K}, we obtain that for all
$x,y\in K\cap\Omega_{2\delta}$ with $
0<\delta<\delta_{\kappa}<1$, 
\begin{equation}\label{eq:interior-holder-K}
|u(x)-u(y)|
\le
B_KL_K^\nu d_\beta(x,y)^\nu .
\end{equation}

It remains to consider points close to the boundary. If
$K\cap\partial\Omega\neq\varnothing$, let $\delta_K>0$ be the constant given by
Proposition~\ref{near boundary K}. If $K\cap\partial\Omega=\varnothing$, choose
$\delta_K>0$ so small that
$d_\beta(K,\partial\Omega)>4\delta_K.$
Fix 
$
0<R_K<
\min\left\{\delta_{\kappa},\frac{\delta_K}{4}
\right\}.
$
Let $x,y\in K$ and set
$d_0:=d_\beta(x,y).$

First assume that $0<d_0<R_K$. Suppose that
$d_\beta(x,\partial\Omega)\ge2d_0$ and  $
d_\beta(y,\partial\Omega)\ge2d_0.$
Then $x,y\in\Omega_{2d_0}$. Applying
\eqref{eq:interior-holder-K} with $\delta=d_0$ gives
$$
|u(x)-u(y)|
\le
B_KL_K^\nu d_\beta(x,y)^\nu.
$$

It remains to consider the case where at least one of the two points is within
distance $2d_0$ from $\partial\Omega$. Up to exchanging $x$ and $y$, assume
$d_\beta(x,\partial\Omega)<2d_0.$
Then
$
d_\beta(y,\partial\Omega)
\le
d_\beta(y,x)+d_\beta(x,\partial\Omega)
<
3d_0.
$
Since $4d_0<\delta_K$, Proposition~\ref{near boundary K}, applied with
$s=4d_0$, gives
$
|u(x)-u(y)|
\le
C_{0}  \mathcal H_\phi\left(\frac{1}{2}d_\beta(K\cap\partial\Omega,S)\right) (4d_0)^{\alpha/2}.
$
Since
$
\nu<\frac{\alpha}{2}
$
and $d_0<1$, we obtain
$$
|u(x)-u(y)|
\le
4^{\frac{\alpha}{2}} C_{0}  \mathcal H_\phi\left(\frac{1}{2}d_\beta(K\cap\partial\Omega,S)\right) d_\beta(x,y)^\nu.
$$

Finally, if $d_\beta(x,y)\ge R_K$, then the boundedness of $u$ gives
$$
|u(x)-u(y)|
\le
2\|u\|_{L^\infty(\Omega)}
\le
2\|u\|_{L^\infty(\Omega)}R_K^{-\nu}
d_\beta(x,y)^\nu.
$$
Combining all cases, we obtain
$$
|u(x)-u(y)|
\le
C_{K,\nu}d_\beta(x,y)^\nu,
\qquad
x,y\in K,
$$
where one may take
$
C_{K,\nu}
:=
\max
\left\{
B_K L_K^{\nu},
4^{\frac{\alpha}{2}} C_{0}  \mathcal H_\phi\left(\frac{1}{2}d_\beta(K\cap\partial\Omega,S)\right), 2\|u\|_{L^\infty(\Omega)}
R_K^{-\nu}
\right\}.
$
The constant depends on the
local geometry near $K$, on the stability data
$p,\|f\|_{L^p(\Omega)}$, and on the lower bound of $\widetilde\rho_E$ on
$\pi^{-1}(K)$. It may blow up when $K$ approaches
$X_{\rm sing}$.
\end{proof}

We now combine the preceding estimates to obtain the main H\"older regularity
result.

\begin{theorem}
\label{thm:finite-singular-boundary-holder}
Let $u$ be the solution given by Theorem~\ref{main thm}. Assume that
$\phi$ is locally $\alpha$-H\"older continuous on
$\partial\Omega\cap X_{\rm reg}$ and $S$ is a holomorphic peak set for
$\Omega$. Assume moreover that the local $\alpha$-H\"older constants of
$\phi$ blow up at most polynomially near $S$: there exist constants
$C>0$, $\sigma\ge0$, and $\varepsilon_0>0$ such that
$$
|\phi(\xi)-\phi(\eta)|
\le
C\,\max\{d_\beta(\xi,S),d_\beta(\eta,S)\}^{-\sigma}
d_\beta(\xi,\eta)^\alpha
$$
whenever $\xi,\eta\in\partial\Omega\cap X_{\rm reg}$ and
$0<\max\{d_\beta(\xi,S),d_\beta(\eta,S)\}<\varepsilon_0$.
Then, for every exponent $\alpha_*$ satisfying
$0<\alpha_*<
\min\left\{
\frac{1}{nq+1},
\frac{\alpha}{2(\sigma+1)}
\right\}$,
the solution $u$ is locally $\alpha_*$-H\"older continuous on
$\Omega\cap X_{\rm reg}$.
\end{theorem}

\begin{proof}
Let $\Gamma_\lambda=\{\xi\in\partial\Omega:\ d_\beta(\xi,S)\ge\lambda\}$. 
For $0<\lambda<\varepsilon_0$, let
$\xi\in\Gamma_\lambda\cap\partial\Omega\cap X_{\rm reg}$ and
$\eta\in\partial\Omega\cap X_{\rm reg}$ with $\xi\ne\eta$.

Assume first that $\max\{d_\beta(\xi,S),d_\beta(\eta,S)\}<\varepsilon_0$.
Then the polynomial blow-up assumption gives
$|\phi(\xi)-\phi(\eta)|
\le
C\,\max\{d_\beta(\xi,S),d_\beta(\eta,S)\}^{-\sigma}
d_\beta(\xi,\eta)^\alpha$. 
Since $\xi\in\Gamma_\lambda$, we have
$\max\{d_\beta(\xi,S),d_\beta(\eta,S)\}\ge d_\beta(\xi,S)\ge\lambda$
and hence
$|\phi(\xi)-\phi(\eta)|
\le
C\lambda^{-\sigma}d_\beta(\xi,\eta)^\alpha$.

Assume now that
$\max\{d_\beta(\xi,S),d_\beta(\eta,S)\}\ge\varepsilon_0$.
If $d_\beta(\xi,S)\ge\varepsilon_0/2$, then
$\xi\in\Gamma_{\varepsilon_0/2}$. Since
$\Gamma_{\varepsilon_0/2}\Subset\partial\Omega\cap X_{\rm reg}$ and
$\phi$ is locally $\alpha$-H\"older continuous on
$\partial\Omega\cap X_{\rm reg}$, as in
Lemma~\ref{lem:local-holder-boundary-constant} there exists a constant
$C_{\varepsilon_0}>0$ such that
$|\phi(\xi)-\phi(\eta)|
\le
C_{\varepsilon_0}d_\beta(\xi,\eta)^\alpha$
for $\xi\in\Gamma_{\varepsilon_0/2}$ and all $\eta\in\partial\Omega$. Hence, for
$0<\lambda<1$,
$|\phi(\xi)-\phi(\eta)|
\le
C_{\varepsilon_0}\lambda^{-\sigma}d_\beta(\xi,\eta)^\alpha$.

It remains to consider the case $d_\beta(\xi,S)<\varepsilon_0/2$. Since
$\max\{d_\beta(\xi,S),d_\beta(\eta,S)\}\ge\varepsilon_0$, we have
$d_\beta(\eta,S)\ge\varepsilon_0$. Therefore
$d_\beta(\xi,\eta)
\ge
d_\beta(\eta,S)-d_\beta(\xi,S)
\ge
\frac{\varepsilon_0}{2}$.
Thus
$$
|\phi(\xi)-\phi(\eta)|
\le
2\|\phi\|_{C^0(\partial\Omega)}
\le
2\|\phi\|_{C^0(\partial\Omega)}
\left(\frac{2}{\varepsilon_0}\right)^\alpha
d_\beta(\xi,\eta)^\alpha
$$
For $0<\lambda<1$, this gives
$|\phi(\xi)-\phi(\eta)|
\le
C'_{\varepsilon_0}\lambda^{-\sigma}d_\beta(\xi,\eta)^\alpha$.

It remains to remove the restriction $\eta\in \partial\Omega\cap X_{\mathrm{reg}}$.
Let now $\xi\in \Gamma_\lambda$ and let $\eta\in S$. Since
$\partial\Omega\cap X_{\mathrm{reg}}$ is dense in $\partial\Omega$, we can choose
a sequence $\eta_j\in \partial\Omega\cap X_{\mathrm{reg}}$ such that
$\eta_j\to \eta$. The estimate proved above, applied to the pair
$(\xi,\eta_j)$, gives
$$
|\phi(\xi)-\phi(\eta_j)|
\le C'\lambda^{-\sigma} d_\beta(\xi,\eta_j)^\alpha ,
$$
where $C'>0$ is independent of $j$ and $\lambda$. Since
$\phi\in C^0(\partial\Omega)$ and $d_\beta(\xi,\eta_j)\to d_\beta(\xi,\eta)$,
letting $j\to+\infty$ yields
$$
|\phi(\xi)-\phi(\eta)|
\le C'\lambda^{-\sigma} d_\beta(\xi,\eta)^\alpha .
$$
Thus the same estimate holds for all $\xi\in \Gamma_\lambda$ and all
$\eta\in \partial\Omega$ with $\xi\ne\eta$.

Combining the preceding cases, we obtain
$\mathcal H_\phi(\lambda)\le B\lambda^{-\sigma}$
for all $0<\lambda<\min \{1,\varepsilon_0\}$.
By Theorem~\ref{thm:holder-peak-set}, for every exponent $\alpha_*$ satisfying
$0<\alpha_*<
\min\left\{
\frac{1}{nq+1},
\frac{\alpha}{2(\sigma+1)}
\right\}$, $u$ is locally $\alpha_*$-H\"older continuous on
$\Omega\cap X_{\rm reg}$.
\end{proof}

Corollary~\ref{finite-cor} is an immediate
consequence of Theorem~\ref{thm:finite-singular-boundary-holder} and
Proposition~\ref{prop:finite-singular-boundary-peak}. The following example shows that the result applies even when the boundary datum is not H\"older continuous at the singular boundary point.

\begin{example}\label{log_example}
Assume that
$S=\partial\Omega\cap X_{\rm sing}=\{a\}.$
Set  $D=\max_{x,y\in \overline\Omega} d_\beta(x,y)$. Define
$$
\phi(x):=\frac{1}{\log\bigl(2D/d_\beta(x,a)\bigr)} \quad\text{for }x\in\partial\Omega\setminus\{a\}, \qquad \phi(a):=0.
$$
Then $\phi\in C^0(\partial\Omega)$. Indeed, as $x\to a$ on
$\partial\Omega\setminus{a}$, one has
$\phi(x)
=
\frac{1}{\log\bigl(2D/d_\beta(x,a)\bigr)}
\rightarrow 0
=
\phi(a).$
Moreover, $\phi$ is locally Lipschitz on $\partial\Omega\cap X_{\rm reg}$. 
On the other hand, $\phi$ is not H\"older continuous of any positive order at
$a$. Indeed, for every $\gamma>0$,
$$
\frac{|\phi(x)-\phi(a)|}{d_\beta(x,a)^\gamma}
=
\frac{1}
{d_\beta(x,a)^\gamma\log\bigl(2D/d_\beta(x,a)\bigr)}
\longrightarrow +\infty
\qquad
\text{as }x\to a.
$$
Thus $\phi$ has only logarithmic continuity at the singular boundary point
$a$.

We now verify the polynomial degeneration condition. Let
$\xi,\eta\in\partial\Omega\cap X_{\rm reg}$ and set
$R:=\max\{d_\beta(\xi,a),d_\beta(\eta,a)\}$.
Choose $\varepsilon<D$. Assume that $0<R<\varepsilon$.

If $d_\beta(\xi,\eta)\ge \frac{R}{2}$, since $\phi$ is bounded on $\partial\Omega$, we have
$$
|\phi(\xi)-\phi(\eta)|
\le
2\|\phi\|_{C^0(\partial\Omega)}
\le
4\|\phi\|_{C^0(\partial\Omega)}
R^{-1}d_\beta(\xi,\eta).
$$

It remains to consider the case
$d_\beta(\xi,\eta)<\frac{R}{2}$. 
By the triangle inequality,
$|d_\beta(\xi,a)-d_\beta(\eta,a)|\le d_\beta(\xi,\eta)$. Hence, 
$\min\{d_\beta(\xi,a),d_\beta(\eta,a)\}
\ge R-d_\beta(\xi,\eta)>\frac{R}{2}$. 
For $0<t<\varepsilon$, set
$h(t):=\frac{1}{\log(2D/t)}$. 
Then $h'(t)=\frac{1}{t\log^2(2D/t)}$. 
For $t\in [R/2,R]\subset (0,\varepsilon)$, we have
$h'(t)
\le
\frac{2}{R\log^2(2D/\varepsilon)}$.
By the mean value theorem and $|\phi(\xi)-\phi(\eta)|
=
|h(d_\beta(\xi,a))-h(d_\beta(\eta,a))|$, we obtain
\begin{align*}
|\phi(\xi)-\phi(\eta)|
\le
\frac{2}{\log^2(2D/\varepsilon)}
R^{-1}|d_\beta(\xi,a)-d_\beta(\eta,a)|
\le
\frac{2}{\log^2(2D/\varepsilon)}
R^{-1}d_\beta(\xi,\eta)
\end{align*}

Therefore, after setting
$A:=
\max\left\{
4\|\phi\|_{C^0(\partial\Omega)},
\frac{2}{\log^2(2D/\varepsilon)}
\right\}$,
we obtain
$$
|\phi(\xi)-\phi(\eta)|
\le
A\max\{d_\beta(\xi,a),d_\beta(\eta,a)\}^{-1}
d_\beta(\xi,\eta)
$$
whenever $\xi,\eta\in\partial\Omega\cap X_{\rm reg}$ and
$0<\max\{d_\beta(\xi,a),d_\beta(\eta,a)\}<\varepsilon$.
Thus $\phi$ satisfies the polynomial degeneration condition in
Theorem~\ref{thm:finite-singular-boundary-holder} with
$\alpha=1$ and $\sigma=1$.
\end{example}

\subsection{Other moduli of continuity}
---We can similarly treat other moduli of continuity. In this final section, we briefly explain the case of ``logarithmic continuity''.
Fix $\nu>0$. After choosing $t_\nu>0$ sufficiently small, the function
$$
\omega_\nu(t):=(\log(1/t))^{-\nu}
\quad 0<t<t_\nu
$$
is increasing and concave. We call $\omega_\nu$, and any positive multiple of
it, a logarithmic modulus of continuity of order $\nu$.

For a general singular boundary set
$S=\partial\Omega\cap X_{\rm sing}$, Proposition~\ref{prop:logarithmic-psh-cutoff}
provides logarithmic plurisubharmonic cutoff functions. In particular, no peak-set condition on $S$ is required. Applying the H\"older
argument with $\omega_\nu$ in place of the H\"older modulus gives the following
logarithmic version; we omit the repetition of the proof.

\begin{theorem}
\label{thm:intro-logarithmic-regularity}
Let $u$ be the solution given by Theorem~\ref{main thm}.
Assume that $\phi$ has local modulus of continuity $\omega_\nu$ on
$\partial\Omega\cap X_{\rm reg}$. Assume moreover that the local
$\omega_\nu$-moduli of $\phi$ blow up at most logarithmically near $S$:
there exist constants $C>0$, $\sigma\ge0$, and $\varepsilon_0>0$ such that
$$
|\phi(\xi)-\phi(\eta)|
\le
C
\bigl(\log(1/\max\{d_\beta(\xi,S),d_\beta(\eta,S)\})\bigr)^\sigma
\omega_\nu(d_\beta(\xi,\eta))
$$
whenever $\xi,\eta\in\partial\Omega\cap X_{\rm reg}$, 
$0<\max\{d_\beta(\xi,S),d_\beta(\eta,S)\}<\varepsilon_0$ and $0<d_\beta(\xi,\eta)< t_\nu.$
Then for every compact set
$K\Subset\overline\Omega\cap X_{\rm reg}$ and every exponent
$0<\nu_*<
\min\left\{
1, \frac{\nu}{\sigma+1}
\right\},$
there exists a constant $C_{K,\nu_*}>0$ such that
$$
|u(x)-u(y)|
\le
\frac{C_{K,\nu_*}}
{\bigl(\log(1/d_\beta(x,y))\bigr)^{\nu_*}}
$$
for all $x,y\in K$ with $d_\beta(x,y)$ sufficiently small.
\end{theorem}

\begin{prop}
\label{prop:logarithmic-psh-cutoff}
Let $S=\partial\Omega \cap X_{\rm sing}$. For $r>0$, recall $\Omega_r^{\rm reg}
=
\{x\in \overline{\Omega}:\ d_\beta(x,X_{\rm sing})\ge r\}.$
Then there exist constants $C_r>0$ and $\lambda_r>0$ such that, for every
$0<\lambda<\lambda_r$, there exists a function
$\chi_\lambda\in{\rm PSH}(\Omega)\cap C^0(\overline\Omega)$ satisfying
$$
-1\le \chi_\lambda\le 0
\;\text{on }\Omega,\quad
\chi_\lambda=-1
\;
\text{on }\Omega\cap N_\lambda S,\quad
\chi_\lambda\ge
-\frac{C_r}{\log(1/\lambda)}
\;
\text{on }\Omega_r^{\rm reg},
$$
where $N_\lambda S:=\{x\in X:\ d_\beta(x,S)<\lambda\}.$
\end{prop}

\begin{proof}
Choose a Stein open neighborhood $U$ of $\overline{\Omega}$ in $X$. Since
$X_{\mathrm{sing}}\cap U$ is a closed analytic subset of $U$, the Oka--Cartan
coherence theorem for ideal sheaves of analytic sets implies that
$\mathcal I_{X_{\mathrm{sing}}\cap U}$ is coherent
(see \cite[Chapter~4.2, Fundamental Theorem]{GrauertRemmert1984}).
Since $U$ is Stein, Cartan's Theorem~A implies that, for every $x\in U$, the
stalk $\mathcal I_{X_{\mathrm{sing}}\cap U,x}$ is generated, as an
$\mathcal O_{U,x}$-module, by the germs at $x$ of global sections of
$\mathcal I_{X_{\mathrm{sing}}\cap U}$
(see \cite[p.~124, Fundamental Theorem]{GrauertRemmert1979}). By the
finite-type lemma for coherent sheaves
\cite[Annex~3.1]{GrauertRemmert1984}, this generation holds in a neighborhood
of each point. Hence, by the compactness of $\overline{\Omega}$, after replacing
$U$ by a smaller Stein neighborhood of $\overline{\Omega}$ if necessary, there
exist holomorphic functions $g_1,\ldots,g_N\in\mathcal O(U)$ which generate
$\mathcal I_{X_{\mathrm{sing}}\cap U}$ on $U$. Consequently,
$
X_{\mathrm{sing}}\cap U
=
\{g_1=\cdots=g_N=0\}.
$
Set
$|g|(x):=\left(\sum_{j=1}^N |g_j(x)|^2\right)^{1/2}.$
Choose $R>0$ such that
$R>\sup_{\overline{\Omega}}|g|.$
Then
$\Phi(x)
:=
\log\frac{|g(x)|}{R}$
is plurisubharmonic on $U$, hence on $\Omega$, and satisfies
$\Phi\leq 0$ on  $\Omega$.

Since each $g_j$ vanishes on $X_{\mathrm{sing}}\cap U$, the local
Lipschitz estimate for holomorphic functions in local embeddings, together
with a finite covering of $\overline{\Omega}$, gives a constant $A>0$ such
that
$|g(x)|\le A\,d_\beta(x,X_{\mathrm{sing}})
$
for all $x\in\overline{\Omega}$.

Since $S\subset X_{\mathrm{sing}}$,
if $x\in \Omega\cap N_{\lambda}S$,
then
$d_\beta(x,X_{\mathrm{sing}})
\leq
d_\beta(x,S)
<
\lambda.$
Therefore,
$|g(x)|\leq A\lambda$
for all $x\in \Omega\cap N_{\lambda}S$.
For
$0<\lambda<\frac{R}{A},$
define
$L_\lambda:=\log\frac{R}{A\lambda}.$
Then
$L_\lambda>0.$
Define
$\chi_\lambda:=\max\left\{\frac{\Phi}{L_\lambda},-1\right\}.$
Then $\chi_\lambda$ is plurisubharmonic on $\Omega$. Moreover, since $\Phi\leq 0$, one has
$-1\leq \chi_\lambda\leq 0$ on  $\Omega$.
If $x\in \Omega\cap N_{\lambda}S,$
then $|g(x)|\leq A\lambda,$
and hence
$\Phi(x)
=
\log\frac{|g(x)|}{R}
\leq
\log\frac{A\lambda}{R}
=
-L_\lambda.$
Consequently,
$\frac{\Phi(x)}{L_\lambda}\leq -1,$
so
$\chi_\lambda(x)=-1$ for all
$x\in \Omega\cap N_{\lambda}S$.

Since
$ \Omega_r^{\rm reg}$
is compact and
$\{|g|=0\}\cap U
=
X_{\mathrm{sing}}\cap U,$
we have
$
m_r
:=
\inf_{ \Omega_r^{\rm reg}} |g|
>
0.
$
Therefore, for every $x\in  \Omega_r^{\rm reg}$,
$
\Phi(x)
=
\log\frac{|g(x)|}{R}
\geq
\log\frac{m_r}{R}.
$
Set $B_r:=\log\frac{R}{m_r}.$
Then
$
B_r>0
$
and
$
\Phi(x)\geq -B_r
$ for all $x\in  \Omega_r^{\rm reg}$.
Hence
$
\frac{\Phi(x)}{L_\lambda}
\geq
-\frac{B_r}{L_\lambda}.
$
After decreasing $\lambda_r>0$ if necessary, we may assume that
$L_\lambda\geq B_r$ for all $0<\lambda<\lambda_r$.
Thus, for $0<\lambda<\lambda_r$,
$\chi_\lambda(x)
=
\max\left\{
\frac{\Phi(x)}{L_\lambda},
-1
\right\}
\ge
-\frac{B_r}{L_\lambda}=\frac{-B_r }{\log (R/A\lambda)}$ for all 
$x\in  \Omega_r^{\rm reg}$.
After taking the constant \(C_r>0\) and  \(\lambda_r>0\) sufficiently
small, we get the estimate 
$
\chi_\lambda(x)\ge -\frac{C_r}{\log(1/\lambda)}$ on $\Omega_r^{\rm reg}$.
\end{proof}

\begin{example}\label{loglog_example}
Assume that
$S=\partial\Omega\cap X_{\rm sing}$
is nonempty, and set
$D:=\max_{x,y\in\overline\Omega}d_\beta(x,y)$.
Define
$$ \phi(x):=\frac{1}{\log\log\bigl(2e^eD/d_\beta(x,S)\bigr)} \quad\text{for }x\in\partial\Omega\setminus S, \qquad \phi|_S:=0 .$$
This function satisfies the assumptions
of Theorem~\ref{thm:intro-logarithmic-regularity} in the case $\sigma\ge\nu$. However, $\phi$ has no logarithmic modulus of continuity
along $S$. More precisely, for every $\gamma>0$, one has
$$
\frac{\phi(x)}
{(\log(1/d_\beta(x,S)))^{-\gamma}}
=
\frac{(\log(1/d_\beta(x,S)))^\gamma}
{\log\circ\log\bigl(2\e^{\e}D/d_\beta(x,S)\bigr)}
\longrightarrow+\infty
\qquad
\text{as } d_\beta(x,S)\to0.
$$
Theorem~\ref{thm:intro-logarithmic-regularity} still yields a local
logarithmic modulus of continuity for the corresponding solution on
$\Omega\cap X_{\rm reg}$.

\end{example}

\end{document}